\documentclass[reqno]{amsart}
\pdfoutput=1

\usepackage{amsmath}
\usepackage{multicol}
\usepackage{mathrsfs}
\usepackage{amssymb}
\usepackage{amsthm}
\usepackage[top=95pt,bottom=80pt,left=85pt,right=85pt,footskip=35pt]{geometry}
\usepackage{xcolor}
\usepackage{enumerate}
\usepackage[all]{xy}
\xyoption{pdf}
\usepackage{tabularx}

\usepackage{makecell}
\usepackage{csquotes}
\usepackage{hyperref}
\usepackage{tikz}
\usetikzlibrary{shapes,positioning,intersections,quotes,calc}
\usepackage{ctable}
\usepackage{fancyhdr,floatpag}  

 \floatpagestyle{fancy}

\numberwithin{equation}{section}

\theoremstyle{definition}
\newtheorem*{Def}{Definition}
\newtheorem{Thm}{Theorem}[section]
\newtheorem{Cor}[Thm]{Corollary}
\newtheorem{Prop}[Thm]{Proposition}
\newtheorem{Lem}[Thm]{Lemma}
\newtheorem{Rmk}[Thm]{Remark}
\newtheorem{Ex}[Thm]{Example}

\newcommand{\PP}{\mathbb{P}}

\newcommand{\QQ}{\mathbb{Q}}

\newcommand{\Oo}{\mathcal{O}}

\newcommand{\Cc}{\mathcal{C}}
\newcommand{\Dd}{\mathcal{D}}
\newcommand{\Ee}{\mathcal{E}}
\newcommand{\Zz}{\mathcal{Z}}

\newcommand{\Bl}{\mathrm{Bl}}
\newcommand{\Sym}{\operatorname{Sym}}

\DeclareFontFamily{U}{mathc}{}
\DeclareFontShape{U}{mathc}{m}{it}%
{<->s*[1.03] mathc10}{}
\DeclareMathAlphabet{\mathcal}{U}{mathc}{m}{it}

\begin{document}

\title{Positivity of the tangent bundle of rational surfaces with nef anticanonical divisor}
\author{Hosung Kim, Jeong-Seop Kim, and Yongnam Lee}

\address {Department of Mathematics\\
Changwon National University\\
20 Changwondaehak-ro, Uichang-gu\\
Changwon-si, Gyeongsangnam-do, 51140 Korea}
\email{hosungkim@changwon.ac.kr}

\address {School of Mathematics\\
Korea Institute for Advanced Study\\
85 Hoegiro, Dongdaemun-gu\\
Seoul, 02455 Korea}
\email{jeongseop@kias.re.kr}

\address {Center for Complex Geometry\\
Institute for Basic Science (IBS)\\
55 Expo-ro, Yuseong-gu\\ 
Daejeon, 34126 Korea, and 
\newline\hspace*{3mm} Department of Mathematical Sciences\\
KAIST\\
291 Daehak-ro, Yuseong-gu\\ 
Daejeon, 34141 Korea}
\email{ynlee@ibs.re.kr}

\thanks{
MSC 2010: 14J26, 14J60\\
Keywords: rational elliptic surface, weak del Pezzo surface, tangent bundle, bigness,  configuration of $(-2)$-curves} 

\begin{abstract}
In this paper, we study the property of bigness of the tangent bundle of a smooth projective rational surface with nef anticanonical divisor.
We first show that the tangent bundle $T_S$ of $S$ is not big if $S$ is a rational elliptic surface.
We then study the property of bigness of the tangent bundle $T_S$ of a weak del Pezzo surface $S$.
When the degree of $S$ is $4$, we completely determine the bigness of the tangent bundle through the configuration of $(-2)$-curves.
When the degree $d$ of $S$ is less than or equal to $3$, we get a partial answer.
In particular, we show that $T_S$ is not big when the number of $(-2)$-curves is less than or equal to $7-d$, and $T_S$ is big when $d=3$ and $S$ has the maximum number of $(-2)$-curves.
The main ingredient of the proof is to produce irreducible effective divisors on $\PP(T_S)$, using Serrano's work on the relative tangent bundle when $S$ has a fibration, or the total dual VMRT associated to a conic fibration on $S$.
\end{abstract}

\maketitle

\section{Introduction}

Throughout the paper, we will work over the field of complex numbers.
Let $X$ be a smooth projective variety.
We say that the tangent bundle $T_X$ of $X$ is pseudo-effective (resp. big) if the tautological class $\Oo_{\PP(T_X)}(1)$ of the projectivized bundle $\PP(T_X)$ is pseudo-effective (resp. big).
Here, $\PP(T_X)$ is the Grothendieck projectivization of $T_X$.
In general, it is difficult to give a numerical or geometrical characterization for the pseudo-effectiveness or the bigness of the tangent bundle, even on surfaces with negative Kodaira dimension.

For surfaces $S$ with nonnegative Kodaira dimension, we have a complete answer on the positivity of the tangent bundle \cite{JLZ}: the tangent bundle $T_S$ is pseudo-effective if and only if the canonical divisor $K_S$ is nef and the second Chern class $c_2(S)$ vanishes. 
This can also be obtained through a nice splitting structure of the tangent bundle of a smooth non-uniruled projective surface $S$ when $T_S$ is pseudo-effective by H\"oring and Peternell \cite{HP}.

However, for surfaces $S$ with negative Kodaira dimension, the known answer is still far from satisfactory.
H\"oring, Liu, and Shao \cite[Theorem~1.4]{HLS22} extended the result of Mallory in \cite{Mal21} and provided a complete answer for del Pezzo surfaces: if $S$ is a del Pezzo surface of degree $d$,~then
\begin{itemize}
\item[(a)] $T_S$ is pseudo-effective if and only if $d \ge 4$, 
\item[(b)] $T_S$ is big if and only if $d\ge 5$.
\end{itemize}
For relatively minimal ruled surfaces, Kim showed that a projective bundle $\PP_C(E)$ over a smooth projective curve $C$ has a big tangent bundle if and only if $E$ is unstable or $C=\PP^1$ \cite{Kim23}.

In this paper, we study the property of the bigness of the tangent bundle $T_S$ of rational surfaces with nef anticanonical divisor.
We first prove the following theorem.

\begin{Thm}\label{rational elliptic}
Let $S$ be a smooth projective rational elliptic surface.
Then the tangent bundle $T_S$ of $S$ is not big.
\end{Thm}

We then study the property of the bigness of $T_S$ of weak del Pezzo surfaces $S$.
By the upper semicontinuity of the dimension of cohomology under specialization, $T_S$ is big for all weak del Pezzo surfaces $S$ of degree $\ge 5$.
For degree $4$, we completely determine the bigness of the tangent bundle through the configuration of $(-2)$-curves in $S$.

\begin{Thm}\label{degree 4}
Let $T_S$ be the tangent bundle of a weak del Pezzo surface $S$ of degree~$4$.
Then $T_S$ is big except when $S$ has one of the following types of configuration of $(-2)$-curves: 
$\emptyset$, $A_1$, $2A_1$ with $8$ lines, $A_2$, and $A_3$ with $4$ lines.
\end{Thm}

If the degree $d$ of $S$ is less than or equal to $3$, then we have a partial answer.
In particular, we show that $T_S$ is not big if the number of $(-2)$-curves is less than or equal to $7-d$.
We note that the maximum number of $(-2)$-curves in a weak del Pezzo surface of degree $d$ is $9-d$.

\begin{Thm}\label{degree 3}
Let $T_S$ be the tangent bundle of a weak del Pezzo surface $S$ of degree~$d\le 3$.
\begin{enumerate}
\item If the number of $(-2)$-curves in $S$ is less than or equal to $7-d$, then $T_S$ is not big.
\item If $d=3$ and $S$ has the maximum number of $(-2)$-curves, then $T_S$ is big.
\end{enumerate}
\end{Thm}
 
The main ingredient of the proof is to produce irreducible effective divisors on $\PP(T_S)$, using Serrano's work on the relative tangent bundle when $S$ has a fibration, or the total dual VMRT associated to a conic fibration on $S$.
We present some criteria to determine the bigness of $T_S$ in Sections 2 and 4.
Theorem~\ref{rational elliptic} will be proved in Section 3.
Theorem~\ref{degree 4} (resp. Theorem~\ref{degree 3}) will be proved in Section~5 (resp. Section~6).
To show that $T_S$ is big for a weak del Pezzo surface $S$ of degree $3$ with type $E_6$, we use the total dual VMRT and a new method.
We pursue this case in Proposition~\ref{E_6}.

Recently, Martin and Stadlmayr \cite{MS} classified all weak del Pezzo surfaces which have a global vector field over an algebraically closed field of arbitrary characteristic.
Our theorems show that the positivity property of the tangent bundle of weak del Pezzo surfaces is not directly related to infinite automorphism groups.
For instance, if $S$ is a weak del Pezzo surface of degree $4$ of type $2A_1$ with $8$ lines, then $T_S$ is not big even though it has an infinite automorphism group (cf. \cite{MS}).
On the other hand, if $S$ is of type $2A_1$ with $9$ lines, then $T_S$ is big even though it only has finite automorphism groups. 

One interesting question is determining the bigness of the tangent bundle of weak del Pezzo surfaces $S$ of degree $d\le 2$ with the maximum number of $(-2)$-curves in $S$.
Similar questions can also be posed for weak del Pezzo surfaces over a field of positive characteristic.

\medskip
\noindent
{\bf Acknowledgements.}
J.-S. Kim is supported by the NRF grant funded by the Korea government (MSIT) RS-2024-00349592.
Y. Lee is supported by the Institute for Basic Science IBS-R032-D1.
The authors would like to thank Igor Dolgachev and DongSeon Hwang for their valuable comments.

\section{Preliminaries}

Let $f: S\to B$ be a fibration from a smooth projective surface onto a smooth projective curve. Here, a fibration means that all fibers are connected. Let 
\[
D:=\sum_{b\in B}f^{*}b-(f^{*}b)_{\rm red} = \sum_{i\in I} (\nu_{i}-1)E_{i} + \sum_{j\in J} (m_{j}-1)F_{j} = E_{0} + \sum_{j\in J} (m_{j}-1)F_{j}
\]
where $E_0: =\sum_{i\in I} (\nu_i-1)E_i$ comes from the non-multiple non-reduced fibres. This decomposition is indeed the Zariski decomposition where $\sum_{j\in J} (m_{j}-1)F_{j}$ is the nef part and $E_{0}$ is the fixed part.

Due to \cite{Ser92} or \cite[Section~3]{Ser96}, we have an exact sequence
\[
0\longrightarrow T_{S/B}\longrightarrow T_{S}\longrightarrow f^{*}T_{B}\longrightarrow\mathcal{F}\longrightarrow0
\]
where $T_{S}$ and $T_{B}$ are the tangent bundles of $S$ and $B$ respectively, $\mathcal{F}$ is a torsion sheaf, and $T_{S/B}$ is the relative tangent sheaf of $f$, which is locally free.
Let $J_{S/B}$ be the torsion-free image of $T_{S}\to f^{*}T_{B}$.
And we have
\[
J_{S/B}=J_{S/B}^{\vee\vee}\otimes\mathcal{I}_\Gamma =K_{S}^{-1}\otimes T_{S/B}^{-1}\otimes\mathcal{I}_{\Gamma}
\]
where $\mathcal{I}_\Gamma$ is an ideal sheaf of a subscheme $\Gamma$ of $S$, which is  supported on the points $s\in S$ such that the reduced structure $f^{-1}(f(s))_{\textup{red}}$ is singular at $s$.
In addition, we have $T_{S/B}=-K_{S}+f^{*}K_{B}+D$. In particular, we have the following short exact sequence
\[\label{eq:blowup_gamma}
\tag{\dag}
0\longrightarrow T_{S/B} \longrightarrow T_{S} \longrightarrow (-f^{*}K_{B}-D)\otimes\mathcal{I}_\Gamma \longrightarrow 0.
\]
Let $Y:=\mathbb{P}_S((-f^{*}K_B-D)\otimes\mathcal{I}_\Gamma)\cong \mathbb{P}_S(\mathcal{I}_{\Gamma})\subseteq \PP(T_S)$, which is a prime divisor on $W: =\PP(T_S)$.
Note that $Y$ is isomorphic to the blow-up of $S$ along the ideal sheaf $\mathcal{I}_\Gamma$ since $\mathcal I_\Gamma$ is locally generated by a regular sequence.
In general, $Y$ is not necessarily smooth.
Let $\Pi: \PP(T_S)\to S$ be the projective bundle associated to $T_S$ and $\zeta$ be the tautological line bundle on $\PP(T_S)$, and denote by ${\rm Exc}(\Pi|_Y)$ the exceptional divisor of $\Pi|_Y$.
By the short exact sequence \eqref{eq:blowup_gamma}, we have $\zeta|_{Y}=\mathcal{O}_{W}(1)|_{Y}=\mathcal{O}_{Y}(1)$ where
\[
\mathcal{O}_Y(1)=\mathcal{O}_Y(-{\rm Exc}(\Pi|_Y))\otimes (\pi|_Y)^{*}(-f^{*}K_B-D).
\]
Moreover, we know that $\zeta-Y\sim\Pi^{*}T_{S/B}$ and $(\Pi|_Y)_{*}\mathcal{O}_Y(1)=(-f^{*}K_B-D)\otimes\mathcal{I}_{\Gamma}$.

If $f: S\to B$ is a sequence of blowing ups of a conic bundle, then the \emph{total dual VMRT} $\breve\Cc$ on $\PP(T_S)$ associated to the family of conics satisfies \cite{HLS22}
\[
\breve\Cc
= \zeta-\Pi^*T_{S/B}.
\]
Therefore, the total dual VMRT $\breve\Cc$ is equal to the prime divisor $Y$.
Hwang and Ramanan \cite{HR04} first introduced the theory of total dual VMRT in a study of Hecke curves on the moduli of vector bundles on a curve.
Later, Occhetta, Sol\'{a} Conde, and Watanabe \cite{OSW16} generalize the theory to the case of minimal rational curves on uniruled varieties.

\begin{Cor}\label{sing} 
Let $Z$ be a smooth projective surface with a conic bundle structure $Z\to B$.
Let $f: S\to Z$ be the blow-up of $Z$ at points $q_1,\,\ldots,\,q_m$.
Assume that the points $q_{i+1},\,\ldots,\,q_m$ lie on singular points of singular fibers of $Z\to B$. Let $E_i$ be the exceptional curves from blowing up of the points $q_i$ for $i=1,\,\ldots,\,m$.
Let $\breve\Cc$ be the total dual VMRT associated to the strict transforms of conic fibers of $Z$.
Then 
\[
\breve\Cc
= \zeta+\Pi^*f^*K_{Z/B}-\Pi^*E_{i+1}-\cdots-\Pi^*E_m.
\]
\end{Cor}

\begin{Lem}[{\cite[Lemma 2.3]{HLS22}}]
\label{big_lemma}
Let $V$ be a vector bundle on $X$ and $D$ be a big divisor on $X$.
If $V\otimes \Oo_X(-D)$ is pseudo-effective, then $V$ is big.
\end{Lem}

\begin{Lem}[{\cite[Lemma 2.2]{KKL}}]
\label{not_big_lemma}Let $X$ be a smooth projective variety.  
Let $\breve\Cc_i$ be divisorial total dual VMRTs on $\PP(T_X)$ for $i=1,\,2,\,\ldots,\,m$.
If they are summed up to
\[
\sum_{i=1}^m \breve\Cc_i
= k\zeta+\Pi^*D
\]
for some $k>0$ and effective divisor $D$ on $X$ (possibly $D=0$), then $T_X$ is not big.
\end{Lem}

\begin{Lem}[{\cite[Lemma 2.4]{HLS22}}]
\label{domination}
Let $f:X\to Z$ be a birational morphism between smooth projective varieties.
If $T_X$ is big, then $T_Z$ is big.
\end{Lem}

\begin{Lem}
\label{specialization}
Let $X_0$ be a smooth projective variety which is a specialization of smooth projective varieties $X_t$, i.e., there is a smooth morphism $f: \mathcal{X}\to B$ over a DVR $B$ such that the general fiber is $X_t$ and the central fiber is $X_0$. 
Then we have the following.
\begin{enumerate}
\item
If $T_{X_t}$ is big for $t\neq 0$, then $T_{X_0}$ is big.
\item
Assume that the irregularity of $X_0$ is zero.
If $T_{X_t}$ is is pseudo-effective for $t\neq 0$, then $T_{X_0}$ is pseudo-effective.
\end{enumerate}
\end{Lem}

\begin{proof}
Proofs are basically obtained by the upper semi-continuity of the dimension of cohomology.
The assumption in (2) is necessary for the extension of a line bundle on $X_0$ to a line bundle on $\mathcal X$.
\end{proof}

We recall Maruyama's elementary transformation.
Let $h\colon X_2\to X_1$ be a blow-up between smooth projective varieties with the exceptional divisor $D$ in $X_2$.
Let $\xi_i$ be the tautological divisor of the projective bundle $\PP(T_{X_i})$ and $\Pi_i: \PP(T_{X_i})\to X_i$  the natural projection.
There is a natural short exact sequence
\[ 0\longrightarrow T_{X_2}\longrightarrow h^{*}T_{X_1} \longrightarrow \iota_{*}T_{D}(D)\longrightarrow 0. \]
Here, $T_D(D)\simeq \mathcal{O}_D(1)$ and $\iota:D\rightarrow X_2$ is the natural inclusion. Let $\widetilde{D}: =\mathbb{P}_D(T_D(D))$ be the projective subbundle of $D': =\mathbb{P}_D(h^{*}T_{X_1})$ obtained by the exact sequence above. 
Applying Maruyama's elementary transformation \cite[Theorem~1.4 and (1.7)]{Mar82}, we have the following commutative diagram, where $\beta$ is the blow-up of $\mathbb{P}(h^{*}T_{X_1})$ along $\widetilde{D}$
and $\alpha$ is the blow-down of $\Bl_{\widetilde{D}}(\mathbb{P}(h^{*}T_{X_1}))$ along the $\beta$-strict transform of $D'$.
\[\xymatrix{
& & \Bl_{\widetilde{D}}(\mathbb{P} (h^{*}T_{X_1}))\ar[dl]_{\beta}\ar[dr]^{\alpha}\\
\mathbb{P}(T_{X_1})\ar[d]^{\pi_1}&\mathbb{P}(h^{*}T_{X_1})\ar@{-->}[rr]\ar[d]^{\widetilde{\pi}_1}\ar[l]_{\widetilde{h}} & & \mathbb{P}(T_{X_2})\ar[d]^{\pi_2}\\
X_1 & X_2\ar@{=}[rr]\ar[l]_{h} & & X_2
}\]
By \cite[Theorem~1.4 and (1.7)]{Mar82}, we have $\beta^{*}\widetilde{h}^{*}\xi_1=\alpha^{*}\xi_2+G$ where $G$ is the $\beta$-exceptional divisor.

\begin{Lem}\label{effetive_divisor_lemma}
Let $Z$ be a smooth projective variety of dimension $n$ and $f: X\to Z$ be the blow-up of $Z$ along a smooth subscheme $S\subset Z$ of codimension $2$ (which is not necessarily irreducible).
We~denote by $\zeta$ (resp. $\xi$) the tautological line bundle on the projective bundle $\Pi: \PP(T_X)\to X$ (resp. $\Lambda: \PP(T_Z)\to Z$) associated to $T_X$ (resp. $T_Z$), and by $E$ the exceptional divisor of $f: X\to Z$.
\begin{itemize}
\item[(1)] 
If $\xi+\Lambda^*D$ is effective on $\PP(T_Z)$ for some divisor $D$ on $Z$, then the divisor
\[
\zeta+\Pi^*f^*D+\Pi^*E
\]
is effective on $\PP(T_X)$.
\item[(2)]
If $h^0(\PP(T_Z),\Oo_{\PP(T_Z)}(m(\xi+\Lambda^*D)))\sim m^{n}$ for some divisor $D$ on $Z$, then the divisor
\[
k\zeta+k\Pi^*f^*D+(k-1)\Pi^*E
\]
is effective on $\PP(T_X)$ for some $k>0$.
\end{itemize}
\end{Lem}

\section{Rational elliptic surfaces}

In this section, we study the positivity of $T_S$ for smooth projective rational elliptic surfaces $S$. 

\begin{proof}[{\bf Proof of Theorem~\ref{rational elliptic}}]
By Lemma~\ref{domination}, it is enough to show that $T_S$ is not big when $S$ is a relatively minimal rational elliptic surface.

Let $S$ be a relatively minimal rational elliptic surface.
As we explained briefly in Section 2, by considering an elliptic fibration $f: S\to B=\PP^1$,
we have the equation $Y=\zeta- \Pi^{*}T_{S/B}$ where $Y$ is a prime divisor on $W: =\PP(T_S)$ and $T_{S/B}=-K_{S}+f^{*}K_{B}+D$ where
\[
D:=\sum_{b\in B}f^{*}b-(f^{*}b)_{\rm red} = \sum_{i\in I} (\nu_{i}-1)E_{i} + \sum_{j\in J} (m_{j}-1)F_{j} = E_{0} + \sum_{j\in J} (m_{j}-1)F_{j}
\] 
where $E_0: =\sum_{i\in I} (\nu_i-1)E_i$ comes from the non-multiple non-reduced fibers (\cite{Ser92} or \cite[Section~3]{Ser96}).

We first consider the case when $S$ has a section.
In this case, $f: S\to \PP^1$ has no multiple fiber, $-K_S=F$ where $F$ is a fiber, and $D=E_{0}=\sum_{i\in I} (\nu_i-1)E_i$ comes from the non-multiple non-reduced fibers.

By using Kodaira's classification on singular elliptic fibers and $\chi_{\rm top}(S)=12$, the non-multiple non-reduced fibers can occur only in one fiber except two $I_0^*$ (cf. \cite[Section 7 in Chapter 5]{BHPV}).
So, except the case of two $I_0^*$, we have $Y= \zeta+ \Pi^*(\text{non-zero effective divisors})$.
Consider the exact sequence of the normal bundle of an elliptic fiber $F$,
\[
0\to T_F=\Oo_F\to T_S|_F \to N_{F|S}=\Oo_F\to 0.
\]
Then the lifting $\widetilde F$ of $F$ to $\PP(T_S)$ satisfies that $\zeta . \widetilde F=0$ and $Y$ is dominated by a family of these lifting elliptic curves.
Then, by the same proof of Lemma 2.2 in \cite{KKL}, $T_S$ is not big.
Now consider the case of two $I_0^*$.
Then we have $Y= \zeta+ \Pi^*(F-E_1-E_2)$ where each $2E_i$ for $i=1, 2$ is the unique non-reduced component in $I_0^*$.
So $Y= \zeta+ \Pi^*(E_1+\text{four reduced components in $I_0^*$}-E_2)$ and $Y= \zeta+ \Pi^*(E_2+\text{four reduced components in $I_0^*$}-E_1)$.
Therefore, $2Y= 2\zeta+\Pi^*(\text{non-zero effective divisors})$, which implies again that $T_S$ is not big.

We next consider the case when $S$ has no section.
Then $S$ has only one multiple fiber $F_1$, and the construction of $S$ is well-known via a Halphen pencil.
In this case, $-K_S=F-(m-1)F_1$ where $m$ is the multiplicity of the multiple fiber $F_1$ and the same $\chi_{\rm top}(S)=12$.
So,
\[
T_{S/B}=-K_S+f^*K_{B}+D=-F+E_0.
\]
Then we have again $Y= \zeta-\Pi^*T_{S/B}=\zeta+F-E_0$, so we get the same equation for $\zeta$ when $S$ has a section.
Therefore, $T_S$ is not big even if $S$ has no section.
\end{proof}

\begin{Rmk}
Since every relatively minimal rational elliptic surface with a section is obtained by blowing up at the unique base point of the anticanonical linear system of a weak del Pezzo surface of degree $1$, and by the upper semicontinuity of the dimension of cohomology under specialization, $T_S$ is not big if $S$ is a rational surface obtained by blowing up at a general point of a weak del Pezzo surface of degree $1$.
Therefore, if $S$ is a general element in the moduli of rational surfaces with nef and not big anticanonical divisor, then $T_S$ is not big.
\end{Rmk}

\section{Weak Del Pezzo surfaces}

In this section, we study some properties of weak del Pezzo surfaces $S$ related to the bigness of the tangent bundle $T_S$.

Let $S$ be a smooth projective surface given by blowing up a set of \emph{bubble points} $p_1,\ldots,p_r$ of $\mathbb P^2$. 
For the notion of bubble points (or bubble cycles) on surfaces, refer to \cite[Section 7.3.2]{Dol12}.
According to the notation, we denote $p_j \succ p_i$ to indicate that $p_j$ is infinitely near to $p_i$, and the blow-up at $p_j\succ p_i$ is given by a successive blow-up; the blow-up at $p_i$, followed by the blow-up at $p_j$.

We denote the divisors on $S$ as follows.
\begin{itemize}
\item
$H$ is the total transform of $\Oo_{\PP^2}(1)$.
\item
$E_i$ is the strict transform of exceptional divisor over $p_i$.
\end{itemize}
We say that a set $\{p_{\iota_1},\ldots, p_{\iota_s}\}$ of distinct points among $\{p_1,\ldots,p_r\}$ is \emph{collinear} if the divisor $H-\sum_{k=1}^sE_{\iota_k}$ is effective on $S$.

\begin{Ex}\label{succesive_blow-up}
Let $p_1,\,\ldots,\,p_6$ be six bubble points in $\PP^2$ given by $p_4\succ p_1$, $p_5\succ p_2$, $p_6\succ p_3$, and $\{p_1,\,p_2,\,p_4\}$, $\{p_2,\,p_3,\,p_5\}$, $\{p_1,\,p_3,\,p_6\}$ are collinear.
Then the blow-up surface $S=\Bl_{p_1,\ldots,p_6}\PP^2$ becomes a smooth projective surface of degree $3$ which has three $(-1)$-curves and six $(-2)$-curves with a configuration of type $3A_2$.
In the figure below, we describe a closed point by a circle, and an infinitely near point by an arrow attached to the associated point.
Moreover, we describe the configuration of $(-1)$ and $(-2)$-curves on $S$ by thin and thick lines, respectively.

\vspace{-0.5em}
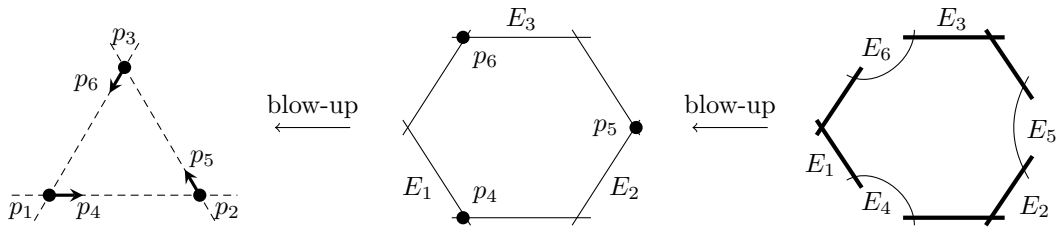
\begin{figure}[ht]
\begin{tikzpicture}[scale=1]
\begin{scope}[xshift=-5.5cm]
\node[circle,fill=black,inner sep=0pt,minimum size=5pt] at (0,0) {};
\node[below left=3pt of {(0,0.1)}] {$p_1$};
\node[below left=3pt of {(0.9,0.1)}] {$p_4$};
\node[circle,fill=black,inner sep=0pt,minimum size=5pt] at (2,0) {};
\node[below right=3pt of {(2,0.1)}] {$p_2$};
\node[above right=10pt of {(1.5,0)}] {$p_5$};
\node[circle,fill=black,inner sep=0pt,minimum size=5pt] at (1,1.7) {};
\node[left=3pt of {(1.4,2.1)}] {$p_3$};
\node[left=3pt of {(0.9,1.5)}] {$p_6$};
\draw[very thick,color=black,-stealth] (0,0)--(0.45,0);
\draw[very thick,color=black,-stealth] (2,0)--(1.8,0.34);
\draw[very thick,color=black,-stealth] (1,1.7)--(0.8,1.36);

\draw[densely dashed,color=black] (-0.5,0)--(2.5,0);
\draw[densely dashed,color=black] (-0.2,-0.34)--(1.2,2.04);
\draw[densely dashed,color=black] (2.2,-0.34)--(0.8,2.04);
\end{scope}

\begin{scope}[xshift=0cm,yshift=-0.3cm]
\def\s{0.45};
\draw (-0.15,0)--(2.15-\s,0);
\draw (-0.15,2.4)--(2.15-\s,2.4);

\draw (0.1,-0.1)--(-0.7-0.1,1.2+0.1);
\draw (-0.7-0.1,1.2-0.1)--(+0.1,2.4+0.1);
\draw (2-\s-0.1,-0.1)--(2-\s+0.7+0.1,1.2+0.1);
\draw (2-\s+0.7+0.1,1.2-0.1)--(2-\s-0.1,2.4+0.1);

\node[circle,fill=black,inner sep=0pt,minimum size=5pt] at (0,0) {};
\node[circle,fill=black,inner sep=0pt,minimum size=5pt] at (2-\s+0.75,1.2) {};
\node[circle,fill=black,inner sep=0pt,minimum size=5pt] at (0,2.4) {};

\node at (-0.6,0.4) {$E_1$};
\node at (2-\s+0.6,0.4) {$E_2$};
\node at (1-\s/2,2.65) {$E_3$};

\node at (0.3,0.3) {$p_4$};
\node at (1.9,1.2) {$p_5$};
\node at (0.3,2.1) {$p_6$};

\draw [->] (-1.5,1.2) -- node [above] {blow-up} (-2.5,1.2);
\draw [->] (2+0.7-\s+0.8+1,1.2) -- node [above] {blow-up} (2+0.7-\s+0.8,1.2);
\end{scope}

\begin{scope}[xshift=5.5cm,yshift=-0.3cm]
\def\s{0.45};
\draw[ultra thick] (0.35,0)--(2.15-\s,0);
\draw[ultra thick] (0.35,2.4)--(2.15-\s,2.4);

\draw[ultra thick] (0.1-0.3,-0.1+0.5)--(-0.7-0.1,1.2+0.1);
\draw[ultra thick] (-0.7-0.1,1.2-0.1)--(+0.1-0.3,1.9+0.1);

\draw[ultra thick] (2-0.1-\s,-0.1)--(2+0.7+0.1 -0.28-\s,1.2+0.1 -0.48);
\draw[ultra thick] (2+0.7+0.1 -0.28-\s,1.2-0.1 +0.48)--(2-0.1-\s,2.4+0.1);

\draw (2+0.7+0.05 -0.28-\s,1.2-0.2 -0.48) to[out=120,in=-120] (2+0.7+0.05 -0.28-\s,1.2+0.2 +0.48);
\draw (-0.3-0.1,0.5) to[out=30,in=90] (0.5,0-0.1);
\draw (-0.3-0.1,2.4-0.5) to[out=-30,in=-90] (0.5,2.4+0.1);

\node at (-0.75,0.7) {$E_1$};
\node at (2.1,0.15) {$E_2$};
\node at (1,2.65) {$E_3$};
\node at (0,0.2) {$E_4$};
\node at (2.2,1.15) {$E_5$};
\node at (0,2.25) {$E_6$};
\end{scope}
\end{tikzpicture}
\vspace{-1em}
\caption{Configuration of $(-1)$ and $(-2)$-curves of type $3A_2$ (degree $3$)}
\label{fig:3A_2}
\vspace{-0.5em}
\end{figure}
\end{Ex}

\begin{Def}
A smooth projective surface $S$ is called \emph{a weak del Pezzo surface} if the canonical divisor $-K_S$ is big and nef.
\end{Def}

\begin{Rmk}
The classification of weak del Pezzo surfaces $S$ through the configuration of $(-2)$-curves and the number of $(-1)$-curves is well-known (cf. \cite{Dol12}).
Let $C$ be an irreducible curve in $S$.
Since $-K_S$ is nef and big, $K_S\cdot C <0$ unless $C$ is a $(-2)$-curve.
So if $C^2\le 0$ then $C$ is either a $(-1)$-curve or a $(-2)$-curve. 

To check the nefness of a line bundle $A$ on a weak del Pezzo surface $S$, it is enough to check $A\cdot C \ge 0$ for all $(-1)$ and $(-2)$-curves $C$ on $S$ by the cone theorem (cf. \cite{KM98}, or [\cite{San14} Proposition 4.1]).
Let $S$ be a weak del Pezzo surface and a line bundle $A=-K_S+\sum a_iF_i$ where $a_i\in\QQ$ and $F_i$ consists of $(-1)$ and $(-2)$-curves on $S$.
Then $A$ is ample on $S$ if $A\cdot C > 0$ for all $(-1)$ and $(-2)$-curves $C$ on $S$ and $A^2 > 0$ by the Nakai-Moishezon criterion and the Kleiman's ampleness criterion (cf. \cite{KM98}).
Since $\zeta$ is relatively ample with respect to the map $\Pi$, $\zeta+\alpha\Pi^*A$ is nef (in fact, it is ample) for sufficiently large $\alpha$ if $A$ is ample.
\end{Rmk}

\begin{Prop}\label{fibration3}
Let $S$ be a weak del Pezzo surface of degree $3$. Let $\ell$ be a line, i.e., $\ell$ is isomorphic to $\PP^1$ and $-K_S.\ell=1$.
Then the linear system $|-K_S-\ell|$ gives a conic pencil on $S$ if $\ell$ does not meet any $(-2)$-curves on $S$. And the linear system $|-K_S-\ell-E|$ gives a conic pencil on $S$ where $E$ is the sum of $(-2)$-curves whose configuration meets $\ell$.
\end{Prop}

\begin{proof}
Suppose that $\ell$ does not meet any $(-2)$-curves.
Then we have clearly $(-K_S-\ell)^2=0$.
We also get $h^0(S, \Oo_S(-K_S-\ell))=2$ by the exact sequence
\[
0 \to \Oo_S(-K_S-\ell) \to \Oo_S(-K_S) \to \Oo_{\ell}(-K_S) \to 0
\]
and the Kawamata-Viehweg vanishing theorem.
We note that $h^0(\ell, \Oo_{\ell}(-K_S))=h^0(\PP^1, \Oo_{\PP^1}(1))=2$ and that $h^1(S, \Oo_S(-K_S-\ell))=h^1(S, \Oo_S(K_S+(-2K_S-\ell)))=0$ because $\Oo_S(-2K_S-\ell)$ is a nef and big divisor
Let $F$ be a general fiber of the linear system $|-K_S-\ell|$.
Then we have $F.(-K_S-\ell)=2$.
So, we prove the first statement.

Now, let $E$ be the sum of $(-2)$-curves whose configuration meets $\ell$.
Then by the same computation we have $(-K_S-\ell)^2=0$ and $h^0(S, \Oo_S(-K_S-\ell))=2$.
But $(-K_S-\ell).E< 0$, so $E$ is the fixed part of the linear system $|-K_S-\ell|$.
And we have $F.(-K_S)=2$ for a general fiber $F$ of the linear system $|-K_S-\ell-E|$.
Therefore, the linear system $|-K_S-\ell-E|$ gives a conic pencil on $S$.
\end{proof}

\begin{Prop}
Let $S$ be the blow-up of $\PP^2$ at a set of bubble points, including four closed points $p_1,\,\ldots,\,,p_4$, no three of which are collinear, and another closed point $p_5$ which does not lie at the intersection of any pair of lines connecting two points among $p_1,\,\ldots,\,p_4$.
Then $T_S$ is not big.
\end{Prop}

\begin{proof}
It suffices to show that $S=\Bl_{p_1,\ldots,p_5}\PP^2$ does not have big $T_S$ due to Lemma \ref{domination}. 
Note that
\[
K_S
=-3H+E_1+E_2+E_3+E_4+E_5.
\]
We define a set of total dual VMRTs on $\PP(T_S)$ as follows.
\begin{itemize}
\item
$\breve\Cc$ is the total dual VMRT associated to the family of strict transform $\ell$ of lines on $\PP^2$ passing through $p_5$.
\item
$\breve\Dd$ is the total dual VMRT associated to the family of strict transform $C$ of conics on $\PP^2$ passing through $p_1$, $p_2$, $p_3$, $p_4$.
\end{itemize}
Then $[\ell]=H-E_5$, $[C]=2H-E_1-E_2-E_3-E_4$, and no non-reduced curve exists in each linear class.
So we have
\begin{align*}
\breve\Cc
&=\zeta+\Pi^*(K_S+2[\ell])= \zeta-\Pi^*H+\Pi^*E_1+\Pi^*E_2+\Pi^*E_3+\Pi^*E_4-\Pi^*E_5,\\
\breve\Dd
&=\zeta+\Pi^*(K_S+2[C])=\zeta+\Pi^*H-\Pi^*E_1-\Pi^*E_2-\Pi^*E_3-\Pi^*E_4+\Pi^*E_5,
\end{align*}
which are summed up to
$
\breve\Cc+\breve\Dd= 2\zeta.$ 
Thus $T_S$ is not big by Lemma \ref{not_big_lemma}.
\end{proof}

\section{Weak del Pezzo surface of degree $4$}

Any weak del Pezzo surface $S$ of degree $4$ can be obtained by blowing up at five bubble points $p_1,\,\ldots,\,p_5$ in $\PP^2$ (cf. \cite[Section 8.6.3]{Dol12}). 

We have the following Figure \ref{Fig-deg4} showing the hierarchy of weak del Pezzo surfaces of degree $4$ with respect to their specializations of the configuration of $(-2)$-curves.
These specializations can be observed through the specializations of the Segre symbol of the quadric pencil induced from weak del Pezzo surfaces of degree $4$ \cite[Table 8.6 and Section 8.6.3]{Dol12}.
More generally, it can be seen by using the Borel-de Sibenthal-Dynkin algorithm \cite[Section 6.4]{DK} or \cite{Ura83}.
In Figure \ref{Fig-deg4}, $2A_1^{(9)}$ means that the configuration of $(-2)$-curves on $S$ is the resolution of two $A_1$-singularities and $S$ has $9$ lines.
Here, by a line, we mean $\mathbb P^1$ on $S$ with $(-K_S)$-degree $1$.
The other notations can be understood similarly.

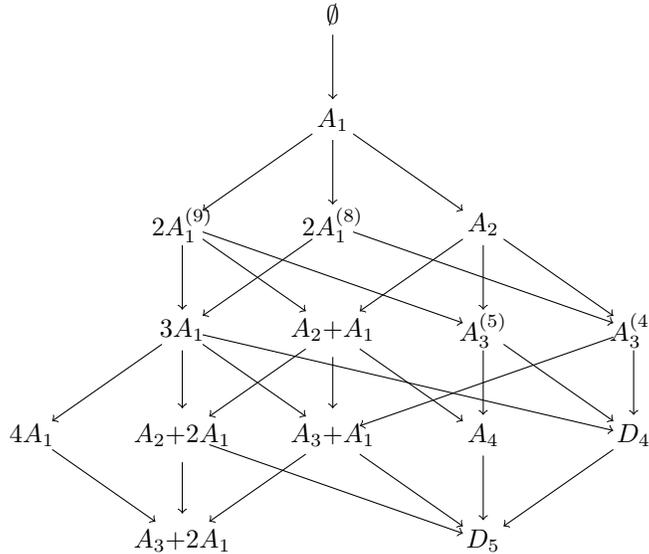
\begin{figure}[ht]
\begin{tikzpicture}[scale=1]

\def\h{1.4}
\def\w{2}
\def\s{15}
\def\l{20}

\node[minimum size=\s,label=center:$\emptyset$] (0) at (0,0) {};

\node[minimum size=\s,label=center:$A_1$] (A1) at (0,-\h*1) {};
\draw [->] (0) -- (A1);

\node[minimum size=\s,label=center:\hyperlink{4-2A_1(9)}{$2A_1^{(9)}$}] (2A1) at (-2*\w/2+\w*0,-\h*2) {};
\draw [->] (A1) -- (2A1);

\node[minimum size=\s,label=center:$2A_1^{(8)}$] (2A1g) at (-2*\w/2+\w*1,-\h*2) {};
\draw [->] (A1) -- (2A1g);

\node[minimum size=\s,label=center:$A_2$] (A2) at (-2*\w/2+\w*2,-\h*2) {};
\draw [->] (A1) -- (A2);

\node[minimum size=\s,label=center:$3A_1$] (3A1) at (-2*\w/2+\w*0,-\h*3) {};
\draw [->] (2A1) -- (3A1);
\draw [->] (2A1g) -- (3A1);

\node[minimum size=\l,label=center:$A_2$+$A_1$] (A1+A2) at (-2*\w/2+\w*1,-\h*3) {};
\draw [->] (2A1) -- (A1+A2);
\draw [->] (A2) -- (A1+A2);

\node[minimum size=\s,label=center:$A_3^{(5)}$] (A3) at (-2*\w/2+\w*2,-\h*3) {};
\draw [->] (2A1) -- (A3);
\draw [->] (A2) -- (A3);

\node[minimum size=\s,label=center:\hyperlink{4-A_3(4)}{$A_3^{(4)}$}] (A3g) at (-2*\w/2+\w*3,-\h*3) {};
\draw [->] (2A1g) -- (A3g);
\draw [->] (A2) -- (A3g);

\node[minimum size=\s,label=center:$4A_1$] (4A1) at (-4*\w/2+\w*0,-\h*4) {};
\draw [->] (3A1) -- (4A1);

\node[minimum size=\l,label=center:$A_2$+$2A_1$] (2A1+A2) at (-4*\w/2+\w*1,-\h*4) {};
\draw [->] (A1+A2) -- (2A1+A2);
\draw [->] (3A1) -- (2A1+A2);

\node[minimum size=\l,label=center:$A_3$+$A_1$] (A1+A3) at (-4*\w/2+\w*2,-\h*4) {};
\draw [->] (A1+A2) -- (A1+A3);
\draw [->] (3A1) -- (A1+A3);
\draw [->] (A3g) -- (A1+A3);

\node[minimum size=\s,label=center:$A_4$] (A4) at (-4*\w/2+\w*3,-\h*4) {};
\draw [->] (A3) -- (A4);
\draw [->] (A1+A2) -- (A4);

\node[minimum size=\s,label=center:$D_4$] (D4) at (-4*\w/2+\w*4,-\h*4) {};
\draw [->] (A3) -- (D4);
\draw [->] (A3g) -- (D4);
\draw [->] (3A1) -- (D4);

\node[minimum size=\l,label=center:$A_3$+$2A_1$] (2A1+A3) at (-4*\w/2+\w*1,-\h*5) {};
\draw [->] (4A1) -- (2A1+A3);
\draw [->] (2A1+A2) -- (2A1+A3);
\draw [->] (A1+A3) -- (2A1+A3);

\node[minimum size=\s,label=center:$D_5$] (D5) at (-4*\w/2+\w*3,-\h*5) {};
\draw [->] (A1+A3) -- (D5);
\draw [->] (A4) -- (D5);
\draw [->] (D4) -- (D5);
\draw [->] (2A1+A2) -- (D5);
\end{tikzpicture}
\vspace{-0.5em}
\caption{Hierarchy between the types of weak del Pezzo surfaces of degree $4$ with respect to specialization}
\label{Fig-deg4}
\end{figure}

\begin{Rmk}
Since the tangent bundle of a del Pezzo surface of degree $4$ is pseudo-effective, the tangent bundle of every weak del Pezzo surface of degree $4$ is also pseudo-effective by Lemma~\ref{specialization}.
\end{Rmk}

\begin{proof}[{\bf Proof of Theorem~\ref{degree 4}}]
From the hierarchy between the types of weak del Pezzo surfaces of degree~$4$ with respect to their specializations and by Lemma~\ref{specialization}, it is enough to show 
\begin{itemize}
\item $T_S$ is big for the case $2A_1^{(9)}$, and
\item $T_S$ is not big for the case $A_3^{(4)}$.
\end{itemize}
Their proofs follow from Lemmas~\ref{lem1_deg4} and \ref{lem2-deg4}.
\end{proof}

Theorem~\ref{degree 4} shows that the positivity property of the tangent bundle of weak del Pezzo surfaces is very delicate. 
If $S$ is a weak del Pezzo surface of degree $4$ of type $2A_1^{(8)}$ or $A_3^{(4)}$, then $T_S$ is not big even though they have infinite automorphism groups (cf. \cite{MS}). On the other hand, if $S$ is of type $2A_1^{(9)}$, then $T_S$ is big even though it has only a finite automorphism group.

\hypertarget{4-2A_1(9)}{
\begin{Lem}\label{lem1_deg4}
Let $S$ be a weak del Pezzo surface of degree $4$ of type $2A_1^{(9)}$.
Then $T_S$ is big.
\end{Lem}
}

\begin{proof}
The surface $S$ can be obtained by blowing up $5$ closed points $p_1,\,\ldots,p_5$ in $\mathbb P^2$
such that $\{p_1,p_3,p_5\}$, $\{p_2,p_4,p_5\}$ are collinear.

\vspace{-0.5em} 
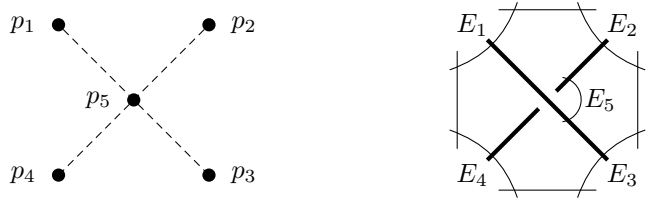
\begin{figure}[ht]
\begin{tikzpicture}[scale=1]
\begin{scope}[xshift=-2.5cm]
\node[circle,fill=black,inner sep=0pt,minimum size=5pt] at (0,0) {};
\node[left=5pt of {(0,0)}] {$p_5$};
\node[circle,fill=black,inner sep=0pt,minimum size=5pt] at (1,1) {};
\node[right=5pt of {(1,1)}] {$p_2$};
\node[circle,fill=black,inner sep=0pt,minimum size=5pt] at (-1,1) {};
\node[left=5pt of {(-1,1)}] {$p_1$};
\node[circle,fill=black,inner sep=0pt,minimum size=5pt] at (-1,-1) {};
\node[left=5pt of {(-1,-1)}] {$p_4$};
\node[circle,fill=black,inner sep=0pt,minimum size=5pt] at (1,-1) {};
\node[right=5pt of {(1,-1)}] {$p_3$};
\draw[densely dashed,color=black] (-1,-1)--(1,1);
\draw[densely dashed,color=black] (1,-1)--(-1,1);
\end{scope}

\begin{scope}[xshift=2.5cm,yshift=-1.2cm]
\draw (-0.15,0)--(1.15,0);
\draw (-0.15,2.4)--(1.15,2.4);

\draw (0.1,-0.1) to[out=110,in=-20] (-0.7-0.1,0.7+0.1);
\draw (-0.7,0.7-0.15)--(-0.7,1.7+0.15);
\draw (-0.7-0.1,1.7-0.1) to[out=20,in=-110] (0+0.1,2.4+0.1);

\draw (1-0.1,-0.1) to[out=70,in=200] (1+0.7+0.1,0.7+0.1);
\draw (1+0.7,0.7-0.15)--(1+0.7,1.7+0.15);
\draw (1+0.7+0.1,1.7-0.1) to[out=160,in=-70] (1-0.1,2.4+0.1);

\draw[ultra thick] (-0.3,0.4)--(1.3,2.0);
\node[circle,fill=white] at (0.5,1.2) {};
\draw[ultra thick] (1.3,0.4)--(-0.3,2.0);
\draw (0.7,1.5) arc[start angle=80, end angle=-80, radius=0.3];

\node at (-0.5,2.2) {$E_1$};
\node at (1.5,2.2) {$E_2$};
\node at (1.5,0.2) {$E_3$};
\node at (-0.5,0.2) {$E_4$};
\node at (1.2,1.2) {$E_5$};
\end{scope}
\end{tikzpicture}
\vspace{-1em} 
\caption{Configuration of $9$ lines and $(-2)$-curves of type $2A_1^{(9)}$ (degree $4$)}
\vspace{-0.5em} 
\end{figure}

\noindent
Then $K_S=-3H+E_1+E_2+E_3+E_4+E_5$.
We define a set of effective divisors on $\PP(T_S)$ as follows.

\begin{itemize}
\item
$\breve\Cc_i$ is the total dual VMRT associated to the family of strict transforms of lines $\ell_i$ on $\PP^2$ passing through $p_i$.
\item 
$\breve\Dd$ is the total dual VMRT associated to the family of strict transforms of conics $C$ on $\PP^2$ passing through $p_1,\,\ldots,\,p_4$.
\item
$\Ee_{13}$ is the pull-back of the strict transform of the line passing through $p_1,\,p_3,\,p_5$.
\item 
$\Ee_{24}$ is the pull-back of the strict transform of the line passing through $p_2,\,p_4,\,p_5$.
\end{itemize}
Then
$[\ell_i]=H-E_i$ with no non-reduced curve,
and $[C]=2H-E_1-E_2-E_3-E_4$ with the non-reduced curve $(H-E_1-E_3-E_5)+(H-E_2-E_4-E_5)+2E_5$.
So we have
\begin{align*}
\breve\Cc_i
&= \zeta+\Pi^*(K_S+2[\ell_1])
= \zeta-\Pi^*H+\Pi^*E_1+\cdots-\Pi^*E_i+\cdots+\Pi^*E_5,\\
\breve\Dd 
&=  \zeta+\Pi^*(K_S+2[C])-\Pi^*E_5
=\zeta+\Pi^*H-\Pi^*E_1-\Pi^*E_2-\Pi^*E_3-\Pi^*E_4,\\
\Ee_{13}
&= \Pi^*H-\Pi^*E_1-\Pi^*E_3-\Pi^*E_5,\\
\Ee_{24}
&= \Pi^*H-\Pi^*E_2-\Pi^*E_4-\Pi^*E_5.
\end{align*}
Here, we use Corollary~\ref{sing} for the equation of $\breve\Dd$.
From these relations, we get the equality
\begin{equation}\label{eq.2A19}
\sum_{i=1}^5\breve\Cc_i
+\breve\Dd+2\Ee_{13}+2\Ee_{24}+\Pi^*E_5
=6\zeta.
\end{equation}
The subcone of $\overline{\text{NE}}(T_S)$ with the rays generated by $\breve\Cc_1$, $\ldots$, $\breve\Cc_5$, $\breve\Dd$, $\Ee_{13}$, $\Ee_{24}$ has full rank $7$.
From this and the equality (\ref{eq.2A19}), we can see that $\zeta$ lies in the interior of $\overline{\text{NE}}(T_S)$.
Thus $T_S$ is big.
\end{proof}

\hypertarget{4-A_3(4)}{
\begin{Lem}\label{lem2-deg4}
Let $S$ be a weak del Pezzo surface of degree $4$ of type $A_3^{(4)}$. Then $T_S$ is not big.
\end{Lem}
}

\begin{proof}
The surface $S$ can be obtained by blowing up 5 bubble points $p_1,\,\ldots,p_5$ in $\mathbb P^2$ such that
$p_3\succ p_2\succ p_1$ and $\{p_1,p_4,p_5\}$ is collinear.

\vspace{-0.5em} 
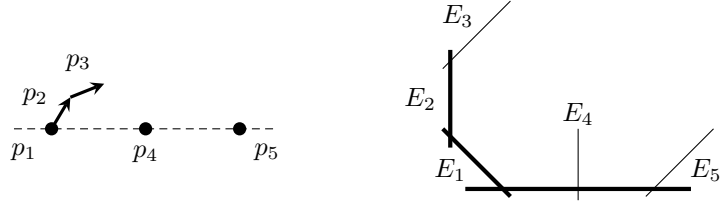
\begin{figure}[ht]
\begin{tikzpicture}[scale=1]
\begin{scope}[xshift=-3cm]
\draw[densely dashed] (-0.5,0)--(3,0);

\node[circle,fill=black,inner sep=0pt,minimum size=5pt] at (0,0) {};
\node[below left=3pt of {(0,0)}] {$p_1$};

\draw[very thick,color=black,-stealth] (0,0)--(0.25,0.42);
\node[left=5pt of {(0.25,0.42)}] {$p_2$};

\draw[very thick,color=black,-stealth] (0.25,0.42)--(0.7,0.6);
\node[above left=2pt of {(0.7,0.6)}] {$p_3$};

\node[circle,fill=black,inner sep=0pt,minimum size=5pt] at (2.5,0) {};
\node[below right=3pt of {(2.5,0)}] {$p_5$};

\node[circle,fill=black,inner sep=0pt,minimum size=5pt] at (1.25,0) {};
\node[below=3pt of {(1.25,0)}] {$p_4$};
\end{scope}

\begin{scope}[xshift=3cm,yshift=-0.8cm]
\draw[ultra thick] (-0.5,0)--(2.5,0);

\draw[ultra thick] (0.1,-0.1)--(-0.7-0.1,0.7+0.1);
\draw[ultra thick] (-0.7,0.7-0.15)--(-0.7,1.7+0.15);
\draw (-0.7-0.1,1.7-0.1)--(0+0.1,2.4+0.1);
\draw (2-0.1,-0.1)--(2+0.7+0.1,0.7+0.1);
\draw (1,-0.15)--(1,0.8);

\node at (-0.7,0.25) {$E_1$};
\node at (-1.1,1.2) {$E_2$};
\node at (-0.6,2.3) {$E_3$};
\node at (1,1) {$E_4$};
\node at (2.7,0.25) {$E_5$};
\end{scope}
\end{tikzpicture}
\vspace{-1em} 
\caption{Configuration of $4$ lines and $(-2)$-curves of type $A_3^{(4)}$ (degree $4$)}
\vspace{-0.5em} 
\end{figure}

\noindent
Then $K_S=-3H+E_1+2E_2+3E_3+E_4+E_5$. 
We define a set of total dual VMRTs on $\PP(T_S)$ as follows.

\begin{itemize}
\item
$\breve\Cc$ is the total dual VMRT associated to the family of strict transforms $\ell$ of lines $\ell$ on $\PP^2$ passing through $p_5$.
\item
$\breve\Dd$ is the total dual VMRT associated to the family of strict transforms $C$ of conics on $\PP^2$ passing through $p_1,\,\ldots,\,p_4$.
\end{itemize}
Then $[\ell]=H-E_5$, $[C]=2H-E_1-2E_2-3E_3-E_4$, and no non-reduced curve exists in each linear class.
So we have 
\begin{align*}
\breve\Cc
&= \zeta+\Pi^*(K_S+2[\ell])=\zeta-\Pi^*H+\Pi^*E_1+2\Pi^*E_2+3\Pi^*E_3+\Pi^*E_4-\Pi^*E_5,\\
\breve\Dd
&= \zeta+\Pi^*(K_S+2[C])=\zeta+\Pi^*H-\Pi^*E_1-2\Pi^*E_2-3\Pi^*E_3-\Pi^*E_4+\Pi^*E_5,
\end{align*}
which are summed up to $\breve\Cc+\breve\Dd=2\zeta$.
Thus $T_S$ is not big by Lemma~\ref{not_big_lemma}.
\end{proof}

\section{Weak del Pezzo surface of degree $\le 3$}

Figure \ref{fig-deg3} shows the hierarchy of weak del Pezzo surfaces of degree $3$ with respect to their specializations of the configuration of $(-2)$-curves.
We can get these specializations through the specializations of the blow-up of six bubble points in $\PP^2$ and the equations of singular cubic surfaces made from contraction of $(-2)$-curves \cite[Section 9.2.2]{Dol12}.
More generally, it can be seen by using the Borel-de Sibenthal-Dynkin algorithm \cite[Section 6.4]{DK} or \cite{Ura83}.

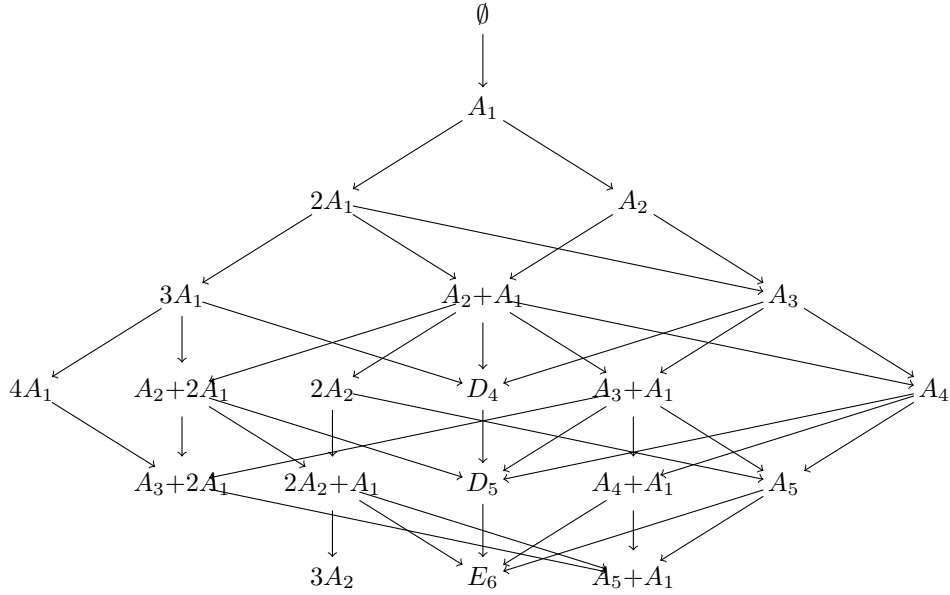
\begin{figure}[ht]
\begin{tikzpicture}[scale=1]
\def\h{1.25}
\def\w{2}
\def\s{15}
\def\l{20}

\node[minimum size=\s,label=center:$\emptyset$] (I) at (0,0) {};

\node[minimum size=\s,label=center:$A_1$] (II) at (0,-\h*1) {};
\draw [->] (I) -- (II);

\node[minimum size=\s,label=center:$A_2$] (III) at (-2*\w/2+\w*2,-\h*2) {};
\draw [->] (II) -- (III);

\node[minimum size=\s,label=center:$2A_1$] (IV) at (-2*\w/2+2*\w*0,-\h*2) {};
\draw [->] (II) -- (IV);

\node[minimum size=\s,label=center:$A_3$] (V) at (-6*\w/2+\w*5,-\h*3) {};
\draw [->] (III) -- (V);
\draw [->] (IV) -- (V);

\node[minimum size=\l,label=center:$A_2$+$A_1$] (VI) at (-6*\w/2+\w*3,-\h*3) {};
\draw [->] (IV) -- (VI);
\draw [->] (III) -- (VI);

\node[minimum size=\s,label=center:$3A_1$] (VIII) at (-6*\w/2+\w*1,-\h*3) {};
\draw [->] (IV) -- (VIII);

\node[minimum size=\s,label=center:\hyperlink{A_4}{$A_4$}] (VII) at (-8*\w/2+\w*7,-\h*4) {};
\draw [->] (V) -- (VII);
\draw [->] (VI) -- (VII);

\node[minimum size=\s,label=center:\hyperlink{4A_1}{$4A_1$}] (XVI) at (-8*\w/2+\w*1,-\h*4) {};
\draw [->] (VIII) -- (XVI);

\node[minimum size=\l,label=center:\hyperlink{A_3+A_1}{$A_3$+$A_1$}] (X) at (-8*\w/2+\w*5,-\h*4) {};
\draw [->] (V) -- (X);
\draw [->] (VI) -- (X);

\node[minimum size=\s,label=center:\hyperlink{2A_2}{$2A_2$}] (IX) at (-8*\w/2+\w*3,-\h*4) {};
\draw [->] (VI) -- (IX);

\node[minimum size=\l,label=center:\hyperlink{A_2+2A_1}{$A_2$+$2A_1$}] (XIII) at (-8*\w/2+\w*2,-\h*4) {};
\draw [->] (VI) -- (XIII);
\draw [->] (VIII) -- (XIII);

\node[minimum size=\s,label=center:\hyperlink{D_4}{$D_4$}] (XII) at (-8*\w/2+\w*4,-\h*4) {};
\draw [->] (VIII) -- (XII);
\draw [->] (VI) -- (XII);
\draw [->] (V) -- (XII);

\node[minimum size=\s,label=center:$A_5$] (XI) at (-6*\w/2+\w*5,-\h*5) {};
\draw [->] (VII) -- (XI);
\draw [->] (X) -- (XI);
\draw [->] (IX) -- (XI);

\node[minimum size=\l,label=center:$A_4$+$A_1$] (XIV) at (-6*\w/2+\w*4,-\h*5) {};
\draw [->] (VII) -- (XIV);
\draw [->] (X) -- (XIV);

\node[minimum size=\s,label=center:$D_5$] (XV) at (-6*\w/2+\w*3,-\h*5) {};
\draw [->] (VII) -- (XV);
\draw [->] (XII) -- (XV);
\draw [->] (XIII) -- (XV);
\draw [->] (X) -- (XV);

\node[minimum size=\l,label=center:\hyperlink{A_3+2A_1}{$A_3$+$2A_1$}] (XVIII) at (-6*\w/2+\w*1,-\h*5) {};
\draw [->] (XVI) -- (XVIII);
\draw [->] (X) -- (XVIII);
\draw [->] (XIII) -- (XVIII);

\node[minimum size=\l,label=center:$2A_2$+$A_1$] (XVII) at (-6*\w/2+\w*2,-\h*5) {};
\draw [->] (IX) -- (XVII);
\draw [->] (XIII) -- (XVII);

\node[minimum size=\l,label=center:$A_5$+$A_1$] (XIX) at (-6*\w/2+\w*4,-\h*6) {};
\draw [->] (XI) -- (XIX);
\draw [->] (XVIII) -- (XIX);
\draw [->] (XVII) -- (XIX);
\draw [->] (XIV) -- (XIX);

\node[minimum size=\s,label=center:\hyperlink{E_6}{$E_6$}] (XX) at (-6*\w/2+\w*3,-\h*6) {};
\draw [->] (XI) -- (XX);
\draw [->] (XVII) -- (XX);
\draw [->] (XV) -- (XX);
\draw [->] (XIV) -- (XX);

\node[minimum size=\s,label=center:$3A_2$] (XXI) at (-6*\w/2+\w*2,-\h*6) {};
\draw [->] (XVII) -- (XXI);
\end{tikzpicture}
\vspace{-0.5em}
\caption{Hierarchy between the types of weak del Pezzo surfaces of degree $3$ with respect to specialization}
\label{fig-deg3}
\vspace{-0.5em}
\end{figure}

In this section, we will prove some lemmas and propositions used to prove Theorem~\ref{degree 3}.  
We also treat some weak del Pezzo surfaces of degree $d\le 3$ with the maximum number of $(-2)$-curves.

\begin{proof}[{\bf Proof of Theorem~\ref{degree 3}}]
(1) Since a weak del Pezzo surface $S$ of degree $d\le 2$ with the $7-d$ number of $(-2)$-curves is obtained by the blow-up of a weak del Pezzo surface $S$ of degree $3$ whose number of $(-2)$-curves is less than or equal to $4$, it is enough to prove that $T_S$ is not big for all weak del Pezzo surfaces $S$ of degree $3$ with the four $(-2)$-curves by Lemma~\ref{domination}.
This will be proven by Lemmas~\ref{2A_2} through \ref{4A_1}.

(2) Bigness of $T_S$ for a weak del Pezzo surface $S$ of degree $3$ of type $A_3+2A_1$ will be proven by Proposition~\ref{A_3+2A_1}.
Then the Borel-de Sibenthal-Dynkin algorithm implies that $T_S$ is big for $S$ of degree $3$ of type $A_5+A_1$.
Bigness of $T_S$ for a weak del Pezzo surface of degree $3$ of type $E_6$ will be proven by Proposition~\ref{E_6}.
If $S$ is a weak del Pezzo surface of degree $3$ of type $3A_2$, then $S$ is toric (cf. \cite[Remark 6]{Der14}, see also Figure \ref{fig:3A_2} in Example \ref{succesive_blow-up}).
Since the tangent bundle of a smooth projective toric variety is big \cite{Hsiao15}, $T_S$ is big.
\end{proof}

\hypertarget{2A_2}{
\begin{Lem}\label{2A_2}
Let $S$ be a weak del Pezzo surface of degree $3$ of type $2A_2$. Then $T_S$ is not big.
\end{Lem}
}

\begin{proof}
The surface $S$ can be obtained by blowing up $6$ bubble points $p_1,\ldots,p_6$ in $\mathbb P^2$ such that
$p_3\succ p_2\succ p_1$ and $p_6\succ p_5\succ p_4$.

\vspace{-0.5em} 
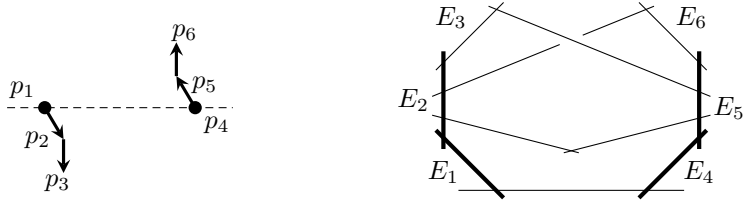
\begin{figure}[ht]
\begin{tikzpicture}[scale=1]
\begin{scope}[xshift=-3cm]
\node[circle,fill=black,inner sep=0pt,minimum size=5pt] at (0,0) {};
\node[above left=0pt of {(0,0)}] {$p_1$};

\draw[densely dashed] (-0.5,0)--(2.5,0);
\draw[very thick,color=black,-stealth] (0,0)--(0.25,-0.42);
\node[below left=0pt of {(0.2,-0.2)}] {$p_2$};

\draw[very thick,color=black,-stealth] (0.25,-0.42)--(0.25,-0.87);
\node[left=1pt of {(0.5,-1)}] {$p_3$};

\node[circle,fill=black,inner sep=0pt,minimum size=5pt] at (2,0) {};
\node[below right=0pt of {(2,0)}] {$p_4$};

\draw[very thick,color=black,-stealth] (2,0)--(1.75,0.42);
\node[right=1pt of {(1.8,0.3)}] {$p_5$};

\draw[very thick,color=black,-stealth] (1.75,0.42)--(1.75,0.87);
\node[right=2pt of {(1.5,1)},color=black] {$p_6$};
\end{scope}

\begin{scope}[xshift=3cm,yshift=-1.1cm]
\draw (-0.5,0)--(2.5,0);

\draw[ultra thick] (0.1,-0.1)--(-0.7-0.1,0.7+0.1);
\draw[ultra thick] (-0.7,0.7-0.15)--(-0.7,1.7+0.15);
\draw (-0.7-0.1,1.7-0.1)--(0+0.1,2.4+0.1);

\draw[ultra thick] (2-0.1,-0.1)--(2+0.7+0.1,0.7+0.1);
\draw[ultra thick] (2+0.7,0.7-0.15)--(2+0.7,1.7+0.15);
\draw (2+0.7+0.1,1.7-0.1)--(2-0.1,2.4+0.1);

\draw (-0.7-0.15,1)--(1.1,0.5);
\draw (2+0.7+0.15,1)--(0.9,0.5);

\draw (-0.7-0.15,1.4-0.15)--(2+0.1,2.35+0.1);
\node[circle,fill=white] at (1,2) {};
\draw (-0.1,2.35+0.1)--(2.7+0.15,1.4-0.15);

\node at (-0.7,0.25) {$E_1$};
\node at (-1.1,1.2) {$E_2$};
\node at (-0.6,2.3) {$E_3$};
\node at (2.7,0.25) {$E_4$};
\node at (3.1,1.1) {$E_5$};
\node at (2.6,2.3) {$E_6$};
\end{scope}
\end{tikzpicture}
\vspace{-1em} 
\caption{Configuration of $7$ lines and $(-2)$-curves of type $2A_2$ (degree $3$)}
\vspace{-0.5em} 
\end{figure}

\noindent
Then $K_S=-3H+E_1+2E_2+3E_3+E_4+2E_5+3E_6$.
We define a set of total dual VMRTs on $\PP(T_S)$ as follows.

\begin{itemize}
\item $\breve\Dd_i$ is the total dual VMRTs associated to the family of strict transforms $C_i$ of conics passing through $p_1,\,p_2,\,p_3,\,p_4$ for $i=1$, and $p_1,\,p_4,\,p_5,\,p_6$ for $i=2$.
\end{itemize}
Then
$[C_1]=2H-E_1-2E_2-3E_3-E_4-E_5-E_6$,
$[C_2]=2H-E_1-E_2-E_3-E_4-2E_5-3E_6$, 
and no non-reduced curve exists in each linear class.
So we have
\begin{align*}
\breve\Dd_1
&= \zeta+\Pi^*(K_S+2[C_1])=\zeta+\Pi^*H-\Pi^*E_1-2\Pi^*E_2-3\Pi^*E_3-\Pi^*E_4+\Pi^*E_6,\\
\breve\Dd_2
&= \zeta+\Pi^*(K_2+2[C_2])=\zeta+\Pi^*H-\Pi^*E_1+\Pi^*E_3-\Pi^*E_4-2\Pi^*E_5-3\Pi^*E_6,
\end{align*}
which are summed up to
\(
\breve\Dd_1+\breve\Dd_2=2\zeta+2\left(\Pi^*H-\Pi^*E_1-\Pi^*E_2-\Pi^*E_3-\Pi^*E_4-\Pi^*E_5-\Pi^*E_6\right).
\)
Thus $T_S$ is not big by Lemma~\ref{not_big_lemma}.
\end{proof}

\hypertarget{A_2+2A_1}{
\begin{Lem}
Let $S$ be a weak del Pezzo surface of degree $3$ of type $A_2+2A_1$.
Then $T_S$ is not big.
\end{Lem}
}

\begin{proof}
The surface $S$ can be obtained by blowing up $6$ bubble points $p_1,\,\ldots,p_6$ in $\mathbb P^2$ such that 
$p_6\succ p_1$, and $\{p_1,\,p_3,\,p_6\}$, $\{p_1,\,p_2,\,p_5\}$, $\{p_2,\,p_3,\,p_4\}$ are collinear.

\vspace{-1em} 
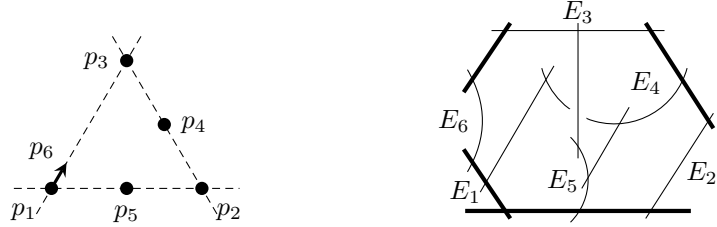
\begin{figure}[ht]
\begin{tikzpicture}[scale=1]

\begin{scope}[xshift=-3cm]
\node[circle,fill=black,inner sep=0pt,minimum size=5pt] at (0,0) {};
\node[below left=3pt of {(0,0)}] {$p_1$};

\node[circle,fill=black,inner sep=0pt,minimum size=5pt] at (2,0) {};
\node[below right=3pt of {(2,0)}] {$p_2$};

\node[circle,fill=black,inner sep=0pt,minimum size=5pt] at (1,1.7) {};
\node[left=3pt of {(1,1.7)}] {$p_3$};

\node[circle,fill=black,inner sep=0pt,minimum size=5pt] at (1.5,0.85) {};
\node[right=3pt of {(1.5,0.85)}] {$p_4$};

\node[circle,fill=black,inner sep=0pt,minimum size=5pt] at (1,0) {};
\node[below=3pt of {(1,0)}] {$p_5$};

\draw[very thick,color=black,-stealth] (0,0)--(0.2,0.34);
\node[left=1pt of {(0.2,0.5)}] {$p_6$};

\draw[densely dashed,color=black] (-0.5,0)--(2.5,0);
\draw[densely dashed,color=black] (-0.2,-0.34)--(1.2,2.04);
\draw[densely dashed,color=black] (2.2,-0.34)--(0.8,2.04);
\end{scope}

\begin{scope}[xshift=3cm,yshift=-0.3cm]
\draw[ultra thick] (-0.5,0)--(2.5,0);
\draw (-0.15,2.4)--(2.15,2.4);

\draw[ultra thick] (0.1,-0.1)--(-0.7-0.1 +0.28,1.2+0.1 -0.48);
\draw[ultra thick] (-0.7-0.1 +0.28,1.2-0.1 +0.48)--(0.1,2.4+0.1);
\draw (2-0.1,-0.1)--(2+0.7+0.1,1.2+0.1);
\draw[ultra thick] (2+0.7+0.1,1.2-0.1)--(2-0.1,2.4+0.1);
\draw (-0.7-0.05 +0.28,1.2-0.2 -0.48) to[out=60,in=-60] (-0.7-0.05 +0.28,1.2+0.2 +0.48);

\draw (1-0.1,-0.15) arc[start angle=-45, end angle=45, radius=0.8];
\draw (2.35+0.1,1.8+0.1) arc[start angle=-15, end angle=-165, radius=1];

\draw (1.05,0.3)--(1.05+0.7*0.9,0.3+1.2*0.9);
\draw (-0.3,0.25)--(-0.3+0.7*1.4,0.25+1.2*1.4);
\node[circle,fill=white] at (0.95,1.2) {};
\draw (1,2.4+0.1)--(1,0.8-0.1);

\node at (-0.5,0.25) {$E_1$};
\node at (2.65,0.5) {$E_2$};
\node at (-0.65,1.2) {$E_6$};
\node at (1,2.65) {$E_3$};
\node at (1.9,1.7) {$E_4$};
\node at (0.8,0.4) {$E_5$};
\end{scope}
\end{tikzpicture}
\vspace{-1em} 
\caption{Configuration of $8$ lines and $(-2)$-curves of type $A_2+2A_1$ (degree $3$)}
\vspace{-0.5em} 
\end{figure}

\noindent
Then $K_S=-3H+E_1+E_2+E_3+E_4+E_5+2E_6$.  
We define a set of total dual VMRTs on $\PP(T_S)$ as follows.
\begin{itemize}
\item
$\breve\Cc$ is the total dual VMRT associated to the family of the strict transforms $\ell$  of lines passing through $p_2$.
\item
$\breve\Dd$ is the total dual VMRT associated to the family of strict transforms $C$ of conics passing through $p_1$, $p_4$, $p_5$, and $p_6$.
\end{itemize}
Then
$[\ell]=H-E_2$,
$[C]=2H-E_1-E_4-E_5-2E_6$,
and no non-reduced curve exists in each linear class.
So we have 
\begin{align*}
\breve\Cc
&=\zeta+\Pi^*(K_S+2[\ell])
=\zeta-\Pi^*H+\Pi^*E_1-\Pi^*E_2+\Pi^*E_3+\Pi^*E_4+\Pi^*E_5+2\Pi^*E_6,\\
\breve\Dd
&= \zeta+\Pi^*(K_S+2[C])
=\zeta+\Pi^*H-\Pi^*E_1+\Pi^*E_2+\Pi^*E_3-\Pi^*E_4-\Pi^*E_5-2\Pi^*E_6,
\end{align*}
which are summed up to
\(
\breve\Cc+\breve\Dd=2\zeta+2\Pi^*E_3.
\)
Thus $T_S$ is not big by Lemma~\ref{not_big_lemma}. 
\end{proof}

\hypertarget{D_4}{
\begin{Lem}\label{D_4}
Let $S$ be a weak del Pezzo surface of degree $3$ of type $D_4$. Then $T_S$ is not big.
\end{Lem}
}

\begin{proof}
The surface $S$ can be obtained by blowing up $6$ bubble points $p_1,\,\ldots,p_6$ in $\mathbb P^2$ such that
$p_4\succ p_1$, $p_5\succ p_2$, $p_6\succ p_3$, and $\{p_1,\,p_2,\,p_3\}$ is collinear.

\begin{figure}[ht]
\begin{tikzpicture}[scale=1]

\begin{scope}[xshift=-3cm]
\draw[densely dashed] (-0.5,0)--(3.5,0);

\node[circle,fill=black,inner sep=0pt,minimum size=5pt] at (0,0) {};
\node[below left=3pt of {(0,0)}] {$p_1$};

\draw[very thick,color=black,-stealth] (0,0)--(0.25,-0.42);
\node[right=2pt of {(0.25,-0.42)}] {$p_4$};

\node[circle,fill=black,inner sep=0pt,minimum size=5pt] at (3,0) {};
\node[below right=3pt of {(3,0)}] {$p_3$};

\draw[very thick,color=black,-stealth] (3,0)--(2.75,0.42);
\node[left=2pt of {(2.75,0.42)}] {$p_6$};

\node[circle,fill=black,inner sep=0pt,minimum size=5pt] at (1.5,0) {};
\node[below=3pt of {(1.5,0)}] {$p_2$};

\draw[very thick,color=black,-stealth] (1.5,0)--(1.5,0.45);
\node[above=2pt of {(1.5,0.45)}] {$p_5$};
\end{scope}

\begin{scope}[xshift=3cm,yshift=-0.5cm]
\draw[ultra thick] (-0.5,0)--(3.5,0);
\draw[ultra thick] (0,-0.15)--(0,0.6+0.1);
\draw[ultra thick] (1.5,-0.15)--(1.5,0.6+0.1);
\draw[ultra thick] (3,-0.15)--(3,0.6+0.1);

\draw (0-0.1,0.6-0.1) to[out=30,in=-30] (0-0.1,1.2+0.1);
\draw (1.5-0.1,0.6-0.1) to[out=30,in=-30] (1.5-0.1,1.2+0.1);
\draw (3-0.1,0.6-0.1) to[out=30,in=-30] (3-0.1,1.2+0.1);

\draw (0-0.1,1.2-0.1)--(1.5+0.1,2+0.1);
\draw (1.5,1.2-0.1)--(1.5,2+0.1);
\draw (3+0.1,1.2-0.1)--(1.5-0.1,2+0.1);

\node at (0-0.3,0.3) {$E_1$};
\node at (1.5-0.3,0.3) {$E_2$};
\node at (3+0.3,0.3) {$E_3$};
\node at (0-0.3,0.9) {$E_4$};
\node at (1.5-0.3,0.9) {$E_5$};
\node at (3+0.4,0.9) {$E_6$};
\end{scope}
\end{tikzpicture}
\vspace{-1em} 
\caption{Configuration of $6$ lines and $(-2)$-curves of type $D_4$ (degree $3$)}
\vspace{-0.5em} 
\end{figure}
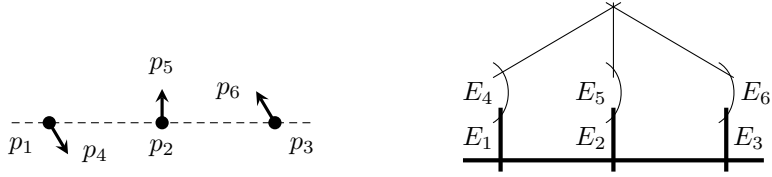

\noindent
Then $K_S=-3H+E_1+E_2+E_3+2E_4+2E_5+2E_6$. 
We define a set of total dual VMRTs on $\PP(T_S)$ as follows.
\begin{itemize}
\item $\breve\Dd_{ij}$ is the total dual VMRT associated to the family of strict transforms $C_{ij}$ of conics passing through $p_i$, $p_{i+3}$, $p_{j}$ and $p_{j+3}$. 
\end{itemize}
Then 
$[C_{12}]=2H-E_1-E_2-2E_4-2E_5$ with the non-reduced curve $2(H-E_1-E_2-\cdots-E_6)+E_1+E_2+2E_3+2E_6$,
$[C_{13}]=2H-E_1-E_3-2E_4-2E_6$ with the non-reduced curve $2(H-E_1-E_2-\cdots-E_6)+E_1+E_3+2E_2+2E_5$, and 
$[C_{23}]=2H-E_2-E_3-2E_5-2E_6$ with the non-reduced curve $2(H-E_1-E_2-\cdots-E_6)+E_2+E_3+2E_1+2E_4$.
So we have
{\small\begin{align*}
\breve\Dd_{12}
&= \zeta+\Pi^*(K_S+2[C_{12}])
-\Pi^*(H-E_1-\cdots-E_6)
-\Pi^*E_3-\Pi^*E_6
=\zeta+\Pi^*E_3-\Pi^*E_4-\Pi^*E_5+2\Pi^*E_6,\\
\breve\Dd_{13}
&= \zeta+\Pi^*(K_S+2[C_{13}])
-\Pi^*(H-E_1-\cdots-E_6)
-\Pi^*E_2-\Pi^*E_5
=\zeta+\Pi^*E_2-\Pi^*E_4+2\Pi^*E_5-\Pi^*E_6,\\
\breve\Dd_{23}
&= \zeta+\Pi^*(K_S+2[C_{23}])
-\Pi^*(H-E_1-\cdots-E_6)
-\Pi^*E_1-\Pi^*E_4
=\zeta+\Pi^*E_1-\Pi^*E_5+2\Pi^*E_4-\Pi^*E_6,
\end{align*}}
which are summed up to $\breve\Dd_{12}+\breve\Dd_{13}+\breve\Dd_{23}=3\zeta+\Pi^*E_1+\Pi^*E_2+\Pi^*E_3$.
Thus $T_S$ is not big by Lemma~\ref{not_big_lemma}.
\end{proof}

\hypertarget{A_4}{
\begin{Lem}
Let $S$ be a weak del Pezzo surface of degree $3$ of type $A_4$. Then $T_S$ is not big.
\end{Lem}
}

\begin{proof} 
The surface $S$ can be obtained by blowing up $6$ bubble points $p_1,\,\ldots,p_6$ in $\mathbb P^2$ such that
$p_3\succ p_2\succ p_1$, $p_5\succ p_4$, and $\{p_1,p_2,p_4\}$ is collinear.

\vspace{-1em} 
\begin{figure}[ht]
\begin{tikzpicture}[scale=1]

\begin{scope}[xshift=-3cm]
\node[circle,fill=black,inner sep=0pt,minimum size=5pt] at (0,0) {};
\node[below left=3pt of {(0,0)}] {$p_1$};
\node[below left=3pt of {(0.7,0)}] {$p_2$};
\node[below left=3pt of {(1,0.9)}] {$p_3$};

\node[circle,fill=black,inner sep=0pt,minimum size=5pt] at (2,0) {};
\node[below right=3pt of {(2,0)}] {$p_4$};
\node[below left=3pt of {(1.8,0.7)}] {$p_5$};

\node[circle,fill=black,inner sep=0pt,minimum size=5pt] at (1,1.7) {};
\node[left=3pt of {(0.9,1.8)}] {$p_6$};

\draw[very thick,color=black,-stealth] (0,0)--(0.5,0.0);

\draw[very thick,color=black,-stealth] (0.5,0)--(0.75,0.42);

\draw[very thick,color=black,-stealth] (2,0)--(1.61,0.225);

\draw[densely dashed,color=black] (-0.5,0)--(2.5,0);
\draw[densely dashed,color=black] (-0.2,-0.34)--(1.2,2.04);
\draw[densely dashed,color=black] (2.2,-0.34)--(0.8,2.04);
\end{scope}

\begin{scope}[xshift=3cm,yshift=-0.2cm]
\draw[ultra thick] (0.35,0)--(2.5,0);
\draw (-0.15,2.4)--(2.15,2.4);

\draw[ultra thick] (0.1 -0.3,-0.1 +0.5)--(-0.7-0.1,1.2+0.1);
\draw (-0.7-0.1,1.2-0.1)--(0.1,2.4+0.1);
\draw[ultra thick] (2-0.1,-0.1)--(2+0.7+0.1,1.2+0.1);
\draw (2+0.7+0.1,1.2-0.1)--(2-0.1,2.4+0.1);
\draw[ultra thick] (-0.3-0.1,0.5) to[out=30,in=90] (0.5,0-0.1);

\draw (0.1,0.3)--(0.6,0.9);
\draw (-0.35-0.15,1.8) to[out=-20,in=200] (1.35+0.1,1.8);
\draw (2.35+0.1,0.6-0.1) to[out=170,in=-90] (1.35-0.1,1.8+0.1);

\node at (-0.75,0.7) {$E_1$};
\node at (0,0.2) {$E_2$};
\node at (0.8,1.1) {$E_3$};
\node at (2.65,0.5) {$E_4$};
\node at (1.9,1.1) {$E_5$};
\node at (1,2.65) {$E_6$};
\end{scope}
\end{tikzpicture}
\vspace{-1em} 
\caption{Configuration of $6$ lines and $(-2)$-curves of type $A_4$ (degree $3$)}
\vspace{-0.5em} 
\end{figure}
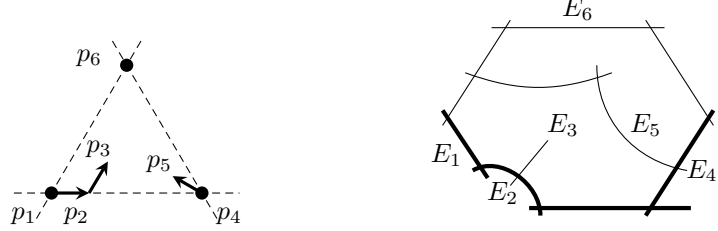

\noindent
Then $K_S=-3H+E_1+2E_2+3E_3+E_4+2E_5+E_6$. 
We define a set of total dual VMRTs on $\PP(T_S)$ as follows.

\begin{itemize}
\item
$\breve\Dd_1$ is the total dual VMRT associated to the family of strict transforms $C_1$ of conics on $\PP^2$ passing through $p_1$, $p_2$, $p_3$, and $p_6$,
\item
$\breve\Dd_2$ is the total dual VMRT associated to the family of strict transforms $C_2$ of conics on $\PP^2$ passing through $p_1$, $p_4$, $p_5$, and $p_6$.
\end{itemize}
Then
$[C_1]=2H-E_1-2E_2-3E_3-E_6$,
$[C_2]=2H-E_1-E_2-E_3-E_4-2E_5-E_6$,
and no non-reduced curve exists in each linear class.
So we have 
\begin{align*}
\breve\Dd_1
&=\zeta+\Pi^*(K_S+2[C_1])
=\zeta+\Pi^*H-\Pi^*E_1-2\Pi^*E_2-3\Pi^*E_3+\Pi^*E_4+2\Pi^*E_5-\Pi^*E_6,\\
\breve\Dd_2
&= \zeta+\Pi^*(K_S+2[C_2])
=\zeta+\Pi^*H-\Pi^*E_1+\Pi^*E_3-\Pi^*E_4-2\Pi^*E_5-\Pi^*E_6,
\end{align*}
which are summed up to $\breve\Dd_1+\breve\Dd_2=2\zeta+2\Pi^*(H-E_1-E_2-E_3-E_6)$.
Thus $T_S$ is not big by Lemma~\ref{not_big_lemma} because $H-E_1-E_2-E_3-E_6$ is the class of the strict transform of the line connecting $p_1$ and $p_6$.  
\end{proof}

\hypertarget{A_3+A_1}{
\begin{Lem}\label{A_3+A_1}
Let $S$ be a weak del Pezzo surface of degree $3$ of type $A_3+A_1$.
Then $T_S$ is not big.
\end{Lem}
}

\begin{proof}      
The surface $S$ can be obtained by blowing up $6$ bubble points $p_1,\,\ldots,p_6$ in $\mathbb P^2$ such that
$p_3\succ p_2\succ p_1$, and $\{p_1,\,p_2,\,p_3\}$, $\{p_1,\,p_4,\,p_5\}$ are collinear.

\vspace{-0.5em} 
\begin{figure}[ht]
\begin{tikzpicture}[scale=1]
\begin{scope}[xshift=-3cm]
\node[circle,fill=black,inner sep=0pt,minimum size=5pt] at (0,0) {};
\node[below left=3pt of {(0,0)}] {$p_1$};
\node[below left=3pt of {(0.2,0.7)}] {$p_2$};
\node[below left=3pt of {(0.5,1.2)}] {$p_3$};

\node[circle,fill=black,inner sep=0pt,minimum size=5pt] at (2,0) {};
\node[below right=3pt of {(2,0)}] {$p_5$};

\node[circle,fill=black,inner sep=0pt,minimum size=5pt] at (1,0) {};
\node[below=3pt of {(1,0)}] {$p_4$};

\node[circle,fill=black,inner sep=0pt,minimum size=5pt] at (1.5,0.85) {};
\node[right=3pt of {(1.5,0.85)}] {$p_6$};

\draw[very thick,color=black,-stealth] (0,0)--(0.25,0.42);

\draw[very thick,color=black,-stealth] (0.25,0.42)--(0.5,0.84);

\draw[densely dashed,color=black] (-0.5,0)--(2.5,0);

\draw[densely dashed,color=black] (-0.2,-0.34)--(1.2,2.04);

\draw[densely dashed,color=black] (2.2,-0.34)--(0.8,2.04);
\end{scope}

\begin{scope}[xshift=3cm,yshift=-0.2cm]
\draw[ultra thick] (-0.5,0)--(2.5,0);

\draw[ultra thick] (0+0.1,0-0.1)--(-0.41/2-0.1,1.26/2+0.1);
\draw[ultra thick] (-0.41+1.41/3+0.05,1.26+1.14/3+0.1)--(1+0.1,2.4+0.1);
\draw (2-0.1,0-0.1)--(2.41+0.05,1.26+0.1);
\draw (2.41+0.1,1.26-0.1)--(1-0.1,2.4+0.1);

\draw[ultra thick] (-0.41/2-0.1,1.26/2-0.1) to[out=30,in=-30] (-0.41/2-0.05,1.26+0.1);
\draw (-0.41/2-0.1,1.26-0.1) to[out=30,in=-75] (-0.41+1.41/3+0.15,1.26+1.14/3+0.3);

\draw (1-0.1,-0.15) arc[start angle=-45, end angle=45, radius=0.7];
\draw (3.41/2+0.1,3.66/2+0.1) arc[start angle=-15, end angle=-105, radius=1];

\draw (-0.2,0.2)--(-0.2+1.1,0.2+1.1);
\draw (1,0.6)--(1,1.3);

\node at (-0.4,0.25) {$E_1$};
\node at (-0.45,0.97) {$E_2$};
\node at (-0.15,1.65) {$E_3$};
\node at (0.8,0.3) {$E_4$};
\node at (2.4,0.5) {$E_5$};
\node at (1.2,1.6) {$E_6$};
\end{scope}
\end{tikzpicture}
\vspace{-1em} 
\caption{Configuration of $7$ lines and $(-2)$-curves of type $A_3+A_1$ (degree $3$)}
\vspace{-0.5em} 
\end{figure}
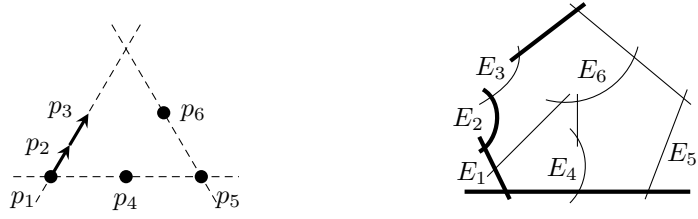

\noindent
Then $K_S=-3H+E_1+2E_2+3E_3+E_4+E_5+E_6$. 
We define a set of total dual VMRTs on $\PP(T_S)$ as follows.

\begin{itemize}
\item
$\breve\Cc$ is the total dual VMRT associated to the family of strict transforms of lines $\ell$ on $\PP^2$ passing through $p_4$.
\item
$\breve\Dd$ is the total dual VMRT associated to the family of strict transforms of conics $C$ on $\PP^2$ passing through $p_1$, $p_2$, $p_5$, and $p_6$.
\end{itemize}
Then
$[\ell]=H-E_4$,
$[C]=2H-E_1-2E_2-2E_3-E_5-E_6$,
and no non-reduced curve exists in each linear class.
So we have 
\begin{align*}
\breve\Cc
&= \zeta+\Pi^*(K_S+2[\ell])
=\zeta-\Pi^*H+\Pi^*E_1+2\Pi^*E_2+3\Pi^*E_3-\Pi^*E_4+\Pi^*E_5+\Pi^*E_6,\\
\breve\Dd
&= \zeta+\Pi^*(K_S+2[C])
=\zeta+\Pi^*H-\Pi^*E_1-2\Pi^*E_2-\Pi^*E_3+\Pi^*E_4-\Pi^*E_5-\Pi^*E_6,
\end{align*}
which are summed up to $\breve\Cc+\breve\Dd=2\zeta+2\Pi^*E_3$.
Thus $T_S$ is not big by Lemma~\ref{not_big_lemma}. 
\end{proof}

\hypertarget{4A_1}{
\begin{Lem}\label{4A_1}
Let $S$ be a weak del Pezzo surface of degree $3$ of type $4A_1$.
Then $T_S$ is not big.
\end{Lem}
}

\begin{proof}
The surface $S$ can be obtained by blowing up $6$ closed points $p_1,\,\ldots,\,p_6$ in $\mathbb P^2$ such that
$\{p_1,\,p_4,\,p_6\}$, $\{p_2,\,p_5,\,p_6\}$, $\{p_1,\,p_3,\,p_5\}$, $\{p_2,\,p_3,\,p_4\}$ are collinear.

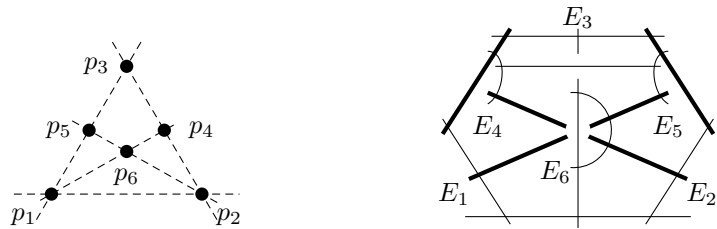
\begin{figure}[ht]
\begin{tikzpicture}[scale=1]

\begin{scope}[xshift=-3cm]
\node[circle,fill=black,inner sep=0pt,minimum size=5pt] at (0,0) {};
\node[below left=3pt of {(0,0)}] {$p_1$};

\node[circle,fill=black,inner sep=0pt,minimum size=5pt] at (2,0) {};
\node[below right=3pt of {(2,0)}] {$p_2$};

\node[circle,fill=black,inner sep=0pt,minimum size=5pt] at (1,1.7) {};
\node[left=3pt of {(1,1.7)}] {$p_3$};

\node[circle,fill=black,inner sep=0pt,minimum size=5pt] at (1.5,0.85) {};
\node[right=3pt of {(1.6,0.85)}] {$p_4$};

\node[circle,fill=black,inner sep=0pt,minimum size=5pt] at (0.5,0.85) {};
\node[left=3pt of {(0.5,0.85)}] {$p_5$};

\node[circle,fill=black,inner sep=0pt,minimum size=5pt] at (1,0.567) {};
\node[below=3pt of {(1,0.57)}] {$p_6$};

\draw[densely dashed,color=black] (-0.5,0)--(2.5,0);
\draw[densely dashed,color=black] (-0.2,-0.34)--(1.2,2.04);
\draw[densely dashed,color=black] (2.2,-0.34)--(0.8,2.04);
\draw[densely dashed,color=black] (-0.45*1/3,-0.25*1/3)--(1.5+0.45*1/3,0.85+0.25*1/3);
\draw[densely dashed,color=black] (0.5-0.45*1/2,0.85+0.25*1/2)--(2+0.45*1/2,-0.25*1/2);
\end{scope}

\begin{scope}[xshift=3cm,yshift=-0.3cm]
\draw (-0.5,0)--(2.5,0);
\draw (-0.15,2.4)--(2.15,2.4);

\draw (0.1,-0.1)--(-0.7-0.1,1.2+0.1);
\draw[ultra thick] (-0.7-0.1,1.2-0.1)--(0.1,2.4+0.1);
\draw (2-0.1,-0.1)--(2+0.7+0.1,1.2+0.1);
\draw[ultra thick] (2+0.7+0.1,1.2-0.1)--(2-0.1,2.4+0.1);

\draw[ultra thick] (-0.35+0.15,1.8-0.15)--(2.35+0.1,0.6-0.1);
\draw[ultra thick] (-0.35-0.1,0.6-0.1)--(2.35-0.15,1.8-0.15);
\node[circle,fill=white] at (1,1.15) {};
\draw (1,-0.15)--(1,2.4+0.1);
\draw (1-0.1,1.5+0.15) arc[start angle=95, end angle=-95, radius=0.5];

\node[circle,fill=white] at (1,2.0) {};
\draw (-0.1,2.0)--(2.1,2.0);
\draw (-0.2,2.2) to[out=0,in=30] (-0.2,1.5);
\draw (2.2,2.2) to[out=180,in=150] (2.2,1.5);

\node at (-0.65,0.3) {$E_1$};
\node at (2.65,0.3) {$E_2$};
\node at (1,2.65) {$E_3$};
\node at (-0.2,1.2) {$E_4$};
\node at (2.2,1.2) {$E_5$};
\node at (0.7,0.6) {$E_6$};
\end{scope}
\end{tikzpicture}
\vspace{-1em} 
\caption{Configuration of $9$ lines and $(-2)$-curves of type $4A_1$ (degree $3$)}
\vspace{-0.5em} 
\end{figure}

\noindent
Then $K_S=-3H+E_1+E_2+\cdots+E_6$. 
We define a set of total dual VMRTs on $\PP(T_S)$ as follows.

\begin{itemize}
\item
$\breve\Cc_i$ is the total dual VMRT associated to the family of strict transforms $\ell_i$ of lines passing through $p_i$ for $i=1,\,\ldots,\,6$.
\item
$\breve\Dd_{ijkl}$ is the total dual VMRT associated to the family of strict transforms $C_{ijkl}$ of conics passing through distinct $p_i$, $p_j$, $p_k$, and $p_l$.
\end{itemize}
Then
$[\ell_i]=H-E_i$ with no non-reduced curve, and
$[C_{ijkl}]=2H-E_i-E_j-E_k-E_l$ with the non-reduced curves $(H-E_{i'}-E_{j'}-E_{m'})+(H-E_{k'}-E_{l'}-E_{m'})+2E_{m'}$ where $\{i',j',k',l'\}=\{i,j,k,l\}$ and $p_{m'}$ is the point of intersection of lines respectively connecting $p_{i'},\,p_{j'}$ and $p_{k'},\,p_{l'}$ if exists.
So we have
\begin{align*}
\breve\Cc_i
&= \zeta+\Pi^*(K_S+2[\ell_i])\}
=\zeta-\Pi^*H+\Pi^*E_1+\cdots-\Pi^*E_i+\cdots+\Pi^*E_6,\\
\breve\Dd_{1236}&= \zeta+\Pi^*(K_S+2[C_{1236}])-\Pi^*E_5-\Pi^*E_6
=\zeta+\Pi^*H-\Pi^*E_1-\Pi^*E_2-\Pi^*E_3-\Pi^*E_6,\\
\breve\Dd_{1245}
&= \zeta+\Pi^*(K_S+2[C_{1245}])-\Pi^*E_3-\Pi^*E_6
=\zeta+\Pi^*H-\Pi^*E_1-\Pi^*E_2-\Pi^*E_4-\Pi^*E_5,\\
\breve\Dd_{3456}
&= \zeta+\Pi^*(K_S+2[C_{3456}])-\Pi^*E_1-\Pi^*E_2
=\zeta+\Pi^*H-\Pi^*E_3-\Pi^*E_4-\Pi^*E_5-\Pi^*E_6,
\end{align*}
which are summed up to $\sum_{i=1}^6 \breve\Cc_i+2\breve\Dd_{1236}+2\breve\Dd_{1245}+2\breve\Dd_{3456}=12\zeta$.
Thus $T_S$ is not big by Lemma~\ref{not_big_lemma}.
\end{proof}

\hypertarget{A_3+2A_1}{
\begin{Prop}\label{A_3+2A_1}
Let $S$ be a weak del Pezzo surface of degree $3$ of type $A_3+2A_1$.
Then $T_S$ is big.
\end{Prop}
}

\begin{proof}
The surface $S$ can be obtained by blowing up $6$ bubble points $p_1,\,\ldots,p_6$ in $\mathbb P^2$ such that
$p_4\succ p_1$, $p_5\succ p_3$, $p_6\succ p_3$, and $\{p_1,p_3,p_4\}$, $\{p_2,p_3,p_5\}$ are collinear.

\vspace{-1em} 
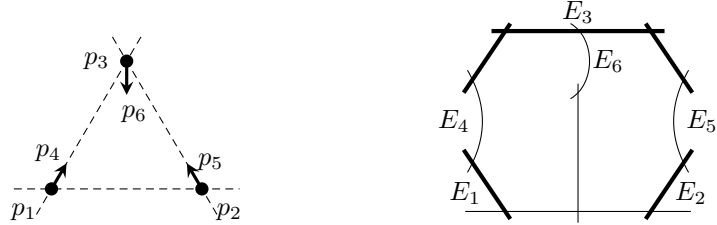
\begin{figure}[ht]
\begin{tikzpicture}[scale=1]
\begin{scope}[xshift=-3cm]
\node[circle,fill=black,inner sep=0pt,minimum size=5pt] at (0,0) {};
\node[below left=3pt of {(0,0)}] {$p_1$};
\node[above left=10pt of {(0.5,0)}] {$p_4$};
\node[circle,fill=black,inner sep=0pt,minimum size=5pt] at (2,0) {};
\node[below right=3pt of {(2,0)}] {$p_2$};
\node[above right=1pt of {(1.8,0.1)}] {$p_5$};
\node[circle,fill=black,inner sep=0pt,minimum size=5pt] at (1,1.7) {};
\node[left=3pt of {(1,1.7)}] {$p_3$};
\node[left=3pt of {(1.5,1)}] {$p_6$};
\draw[very thick,color=black,-stealth] (0,0)--(0.2,0.34);
\draw[very thick,color=black,-stealth] (2,0)--(1.8,0.34);
\draw[very thick,color=black,-stealth] (1,1.7)--(1,1.25);

\draw[densely dashed,color=black] (-0.5,0)--(2.5,0);
\draw[densely dashed,color=black] (-0.2,-0.34)--(1.2,2.04);
\draw[densely dashed,color=black] (2.2,-0.34)--(0.8,2.04);
\end{scope}

\begin{scope}[xshift=3cm,yshift=-0.3cm]
\draw (-0.5,0)--(2.5,0);
\draw[ultra thick] (-0.15,2.4)--(2.15,2.4);

\draw[ultra thick] (0.1,-0.1)--(-0.7-0.1 +0.28,1.2+0.1 -0.48);
\draw[ultra thick] (-0.7-0.1 +0.28,1.2-0.1 +0.48)--(0.1,2.4+0.1);
\draw[ultra thick] (2-0.1,-0.1)--(2+0.7+0.1 -0.28,1.2+0.1 -0.48);
\draw[ultra thick] (2+0.7+0.1 -0.28,1.2-0.1 +0.48)--(2-0.1,2.4+0.1);
\draw (-0.7-0.05 +0.28,1.2-0.2 -0.48) to[out=60,in=-60] (-0.7-0.05 +0.28,1.2+0.2 +0.48);
\draw (2+0.7+0.05 -0.28,1.2-0.2 -0.48) to[out=120,in=-120] (2+0.7+0.05 -0.28,1.2+0.2 +0.48);

\draw (1-0.1,2.4+0.1) to[out=-30,in=30] (1-0.1,1.6-0.1);
\draw (1,-0.15)--(1,1.7);

\node at (-0.5,0.25) {$E_1$};
\node at (2.5,0.25) {$E_2$};
\node at (1,2.65) {$E_3$};
\node at (-0.65,1.2) {$E_4$};
\node at (2.65,1.2) {$E_5$};
\node at (1.4,2) {$E_6$};
\end{scope}
\end{tikzpicture}
\vspace{-1em} 
\caption{Configuration of $5$ lines and $(-2)$-curves of type $A_3+2A_1$ (degree $3$)}
\vspace{-0.5em} 
\end{figure}

\noindent
Then $K_S=-3H+E_1+E_2+E_3+2E_4+2E_5+2E_6$. 
We define a set of effective divisors on $\PP(T_S)$ as follows.

\begin{itemize}
\item
$\breve\Cc_i$ is the total dual VMRT associated to the family of strict transforms $\ell_i$ of lines passing through $p_i$ for $i=1,\,2,\,3$.
\item
$\breve\Dd_{12}$ is the total dual VMRT associated to the family of strict transforms $C_{12}$ of conics passing through $p_1,\,p_4,\,p_2,\,p_5$.
\item
$\breve\Dd_{3}$ is the total dual VMRT associated to the strict transforms $C_3$ of conics passing through $p_1,\,p_2,\,p_3,\,p_6$.
\item
$\Ee_{ij}$ is the pull-back of the strict transform of the line connecting $p_i$ and $p_j$.
\end{itemize}
Then $[\ell_i]=H-E_i-E_{i+3}$ with no non-reduced curve except the case $i=3$ where the linear class has $(H-E_1-E_3-2E_4-E_6)+E_1+2E_4$ and $(H-E_2-E_3-2E_5-E_6)+E_2+2E_5$ as non-reduced curves, $[C_{12}]=2H-E_1-E_2-2E_4-2E_5$ with the non-reduced curves $2(H-E_1-E_2-E_4-E_5)+E_1+E_2$, $(H-E_1-E_3-E_4-E_6)+(H-E_2-E_3-E_5-E_6)+2E_5+2E_6$, and $[C_3]=2H-E_1-E_2-E_3-E_4-E_5-2E_6$ with no non-reduced curve.
So we have 
\begin{align*}
\breve\Cc_1
&=\zeta+\Pi^*(K_S+2[\ell_1])
=\zeta-\Pi^*H-\Pi^*E_1+\Pi^*E_2+\Pi^*E_3+2\Pi^*E_5+2\Pi^*E_6,\\
\breve\Cc_2
&= \zeta+\Pi^*(K_S+2[\ell_2])
=\zeta-\Pi^*H+\Pi^*E_1-\Pi^*E_2+\Pi^*E_3+2\Pi^*E_4+2\Pi^*E_6,\\
\breve\Cc_3
&= \zeta+\Pi^*(K_S+2[\ell_3])-\Pi^*E_4-\Pi^*E_5
=\zeta-\Pi^*H+\Pi^*E_1+\Pi^*E_2-\Pi^*E_3+\Pi^*E_4+\Pi^*E_5,\\
\breve\Dd_{12}
&= \zeta+\Pi^*(K_S+2[C_{12}])
-\Pi^*(H-E_1-E_2-E_4-E_5)-\Pi^*E_5-\Pi^*E_6
=\zeta-\Pi^*E_4-\Pi^*E_5+\Pi^*E_6,\\
\breve\Dd_3
&=\zeta+\Pi^*(K_S+2[C_3])
=\zeta+\Pi^*H-\Pi^*E_1-\Pi^*E_2-\Pi^*E_3-2\Pi^*E_6,\\
\Ee_{12}
&= \Pi^*H-\Pi^*E_1-\Pi^*E_2-\Pi^*E_4-\Pi^*E_5,\\
\Ee_{13}
&= \Pi^*H-\Pi^*E_1-\Pi^*E_3-2\Pi^*E_4-\Pi^*E_6,\\
\Ee_{23}
&= \Pi^*H-\Pi^*E_2-\Pi^*E_3-2\Pi^*E_5-\Pi^*E_6,
\end{align*}
which yields the following equality.
\begin{align*}
\breve\Cc_1+\breve\Cc_2+4\breve\Cc_3+2\breve\Dd_{12}+2\breve\Dd_{3}+2\Ee_{12}+\Ee_{13}+\Ee_{23}+\Pi^*E_1+\Pi^*E_2+6\Pi^*E_3
&=10\zeta
\end{align*}
By the same argument as in the proof of Lemma~\ref{lem1_deg4}, we can show that $T_S$ is big.
\end{proof}

\begin{Lem}\label{big_anticanonical_model}
Let $S$ be a weak del Pezzo surface and $D$ be the union of $(-2)$-curves on $S$.
Assume that $m\zeta-\Pi^*D$ is effective on $\PP(T_S)$ for some $m>0$.
If the orbifold $\Zz\to Z$ associated to the anticanonical model $S\to Z$ of $S$ has a positive second Segre class $c_1(\Zz)^2-c_2(\Zz)>0$, then the tangent bundle $T_S$ of $S$ is big.
\end{Lem}

\begin{proof}
First, we can observe that $\Sym^k[T_S\otimes \Oo_S(D)]$ is a subbundle of $\Sym^{(m+1)k} T_S$;
since $k(m\zeta-\Pi^*D)$ is effective, there exists the following exact sequence on $\PP(T_S)$.
\[
0
\to \Oo_{\PP(T_S)}(k(\zeta+\Pi^*D))
\to \Oo_{\PP(T_S)}((m+1)k\zeta)
\to \cdots
\]
By pushing forward the exact sequence, we obtain the following exact sequence on $S$.
\[
0
\to \Sym^k [T_S\otimes\Oo_S(D)]
\to \Sym^{(m+1)k} T_S
\to \cdots
\]

Moreover, we can observe that $\Sym^k[T_S(-\log D)\otimes \Oo_S(D)]$ is a subsheaf of $\Sym^{(m+1)k} T_S$; for the $(-2)$-curves $E_i$ on $S$ for $i=1,\,\ldots,\,l$, let $D=\sum_{i=1}^l E_i$.
Then there exists the exact sequence
\[
0
\to T_S(-\log D)
\to T_S
\to \Oo_{E_1}(E_1)\oplus\cdots\oplus\Oo_{E_l}(E_l)
\to 0
\]
involving the logarithmic tangent sheaf $T_S(-\log D)$.
By twisting $\Oo_S(D)$ and taking symmetric power to the exact sequence, we obtain the following exact sequence on $S$.
\[
0
\to \Sym^k [T_S(-\log D)\otimes \Oo_S(D)]
\to \Sym^k [T_S\otimes \Oo_S(D)]
\to \mathcal{Q}
\to 0
\]
As $\Sym^k [T_S\otimes \Oo_S(D)]$ is a subsheaf of $\Sym^{(m+1)k} T_S$, $\Sym^k [T_S(-\log D)\otimes \Oo_S(D)]$ is also a subsheaf of $\Sym^{(m+1)k} T_S$.
Thus $h^0(S,\Sym^k[T_S(-\log D)\otimes \Oo_S(D)])\sim k^3$ implies $h^0(S,\Sym^{(m+1)k}T_S)\sim k^3$.
In other words, if the vector bundle $T_S(-\log D)\otimes \Oo_S(D)$ is big, then $T_S$ is big.

By considering the mixed Hodge structure on $S\setminus D$, we have the Poincar\'e residue exact sequence
\[
0
\to \bigwedge^2\Omega_S
\to \bigwedge^2\Omega_S(\log D)
\xrightarrow{\rm res} \Omega_D\to 0.
\]
Since $D$ is a simple normal crossing divisor, $\Omega_D=\Oo_D(K_S+D)$.
Therefore, $\wedge^2\Omega_S(\log D)=K_S+D$, and it implies that $T_S(-\log D)\otimes \Oo_S(D)=\Omega_S(\log D)\otimes \Oo_S(-K_S)$ and
\[
H^0(S\setminus D,\Sym^m \Omega_S)
\cong H^0(S,\Sym^m \Omega_S(\log D))
\]
due to the extension property.
By twisting $\Oo_S(-mK_S)$, we have
\[
H^0(S\setminus D,\Sym^m T_S)
\cong H^0(S, \Sym^m [T_S(-\log D)\otimes \Oo_S(D)]).
\]
On the other hand, we have
\[
H^0(S\setminus D,\Sym^m \Omega_S)\cong H^0(\mathcal{Z}, \Sym^m \Omega_{\mathcal{Z}}).
\]
Thus
\begin{align*}
h^0(S\setminus D,\Sym^m T_S)
&=h^0(S\setminus D,\Sym^m \Omega_S\otimes \Oo_S(-mK_S))\\
&= h^0(\mathcal{Z}, \Sym^m\Omega_Z\otimes \Oo_\mathcal{Z}(-mK_{\mathcal{Z}}))
=h^0(\mathcal{Z}, \Sym^m T_\mathcal{Z}).
\end{align*}
Since $-K_{\mathcal{Z}}$ is an integral Weil and $\QQ$-Cartier divisor, we can deduce that
\[
h^2(\mathcal{Z},\Sym^m T_{\mathcal{Z}})
=h^0(\mathcal{Z},\Sym^m T_{\mathcal{Z}}\otimes \Oo_\mathcal{Z}((m+1)K_\mathcal{Z}))
\leq h^0(\mathcal{Z},\Sym^m T_{\mathcal{Z}}).
\]
Therefore, $T_\mathcal{Z}$ is big if and only if the orbifold Chern classes satisfy $c_1(\Zz)^2-c_2(\Zz)>0$ because
\[
\chi(\mathcal{Z},\Sym^m T_{\mathcal{Z}})
=\frac{1}{6}(c_1(\Zz)^2-c_2(\Zz))\,m^3+O(m^2),
\]
and it implies not only the bigness of $T_S(-\log D)\otimes \Oo_S(D)$ but also the bigness of $T_S$.
\end{proof}

\hypertarget{E_6}{
\begin{Prop}\label{E_6}
Let $S$ be a weak del Pezzo surface of degree $3$ of type $E_6$.
Then $T_S$ is big.
\end{Prop}
}

\begin{proof}
The surface $S$ can be obtained by blowing up $6$ bubble points $p_1,\,\ldots,p_6$ in $\mathbb P^2$ such that
$p_6\succ p_5\succ \cdots p_1$ and $\{p_1,p_2,p_3\}$ is collinear.
\vspace{-1em}

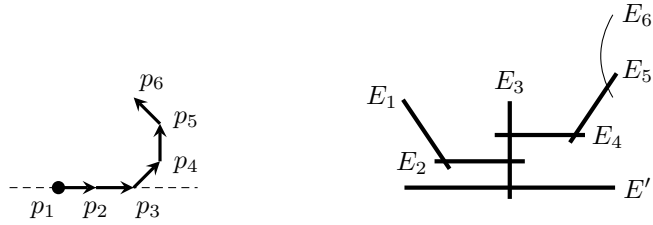
\begin{figure}[h]
\begin{tikzpicture}[scale=1]

\begin{scope}[xshift=-2.5cm]
\draw[densely dashed] (-0.65,0)--(1.85,0);

\node[circle,fill=black,inner sep=0pt,minimum size=5pt] at (0,0) {};
\node at (-0.2,-0.3) {$p_1$};
\node at (0.5,-0.3) {$p_2$};
\node at (1.2,-0.3) {$p_3$};
\node at (1.7,0.3) {$p_4$};
\node at (1.7,0.9) {$p_5$};
\node at (1.25,1.4) {$p_6$};

\draw[very thick,color=black,-stealth] (0,0)--(0.5,0.0);

\draw[very thick,color=black,-stealth] (0.5,0)--(1,0.0);

\draw[very thick,color=black,-stealth] (1,0)--(1.35,0.35);

\draw[very thick,color=black,-stealth] (1.35,0.35)--(1.35,0.85);

\draw[very thick,color=black,-stealth] (1.35,0.85)--(1,1.2);
\end{scope}

\begin{scope}[xshift=2.5cm,yshift=0cm]
\draw[ultra thick] (-0.4,0)--(2.4,0);
\draw[ultra thick] (0,0.35)--(1.2,0.35);
\draw[ultra thick] (0.8,0.7)--(2,0.7);
\draw[ultra thick] (0.2,0.25)--(-0.6-0.1 +0.28,1.55+0.1 -0.48);
\draw[ultra thick] (1.8,0.6)--(2.6+0.1 -0.28,1.9+0.1 -0.48);
\draw[ultra thick] (1,-0.15)--(1,1.15);
\draw (2+0.6+0.05 -0.28,1.9-0.7) to[out=120,in=-120] (2+0.6+0.05 -0.28,1.9+0.4);

\node at (-0.7,1.2) {$E_1$};
\node at (-0.3,0.33) {$E_2$};
\node at (1,1.4) {$E_3$};
\node at (2.3,0.68) {$E_4$};
\node at (2.7,1.55) {$E_5$};
\node at (2.7,2.3) {$E_6$};
\node at (2.7,0) {$E'$};
\end{scope}

\end{tikzpicture}
\vspace{-1em} 
\caption{Configuration of $1$ line and $(-2)$-curves of type $E_6$ (degree $3$)}
\vspace{-1em}
\end{figure}

\noindent
Then $K_S=-3H+E_1+E_2+E_3+2E_4+2E_5+2E_6$. 
We define a set of effective divisors on $\PP(T_S)$ as follows.

\begin{itemize}
\item
$\breve\Cc$ is the total dual VMRT associated to the family of strict transforms $\ell$ of lines passing through $p_1$.
\end{itemize}
Then $[\ell]=H-E_1-E_2-\cdots-E_6$ with the non-reduced curve $(H-E_1-2E_2-3E_3-\cdots-3E_6)+E_2+2E_3+\cdots+2E_6$.
So we have
\begin{align*}
\breve\Cc
&=\zeta+\Pi^*(K_S+2[\ell])-\Pi^*E_3-\cdots-\Pi^*E_6
=\zeta-\Pi^*H-\Pi^*E_1+\Pi^*E_4+2\Pi^*E_5+3\Pi^*E_6.
\end{align*}
Note that
\[
\breve\Cc+\Pi^*E_1+\cdots+\Pi^*E_5+\Pi^*E'
=\zeta-\Pi^*E_1-\Pi^*E_2-2\Pi^*E_3-\Pi^*E_4
\]
where $E'$ is the strict transform of the line passing through $p_1$, $p_2$, and $p_3$.
In particular, $\zeta-\Pi^*D$ is effective on $\PP(T_X)$ where $D=E_1+\cdots+E_5+E'$ is the union of $(-2)$-curves on $S$.
By the formula \cite[Proposition 13]{RR13},
\begin{align*}
c_1(\Zz)^2
&=c_1(S)^2
=3,\\
c_2(\Zz)\ 
&=c_2(S)-(6+1)\cdot 1 + \frac{1}{24}
=\frac{49}{24},
\end{align*}
so we have the desired claim due to Lemma \ref{big_anticanonical_model}.
\end{proof}

The following proposition and the previous remark show that $T_S$ is not necessary to be big even if $S$ has the maximum number of $(-2)$-curves when the degree is $d=2$ or $1$.

\begin{Prop}\label{2A_3+A_1}
Let $S$ be a weak del Pezzo surface of degree $2$ of type $2A_3+A_1$. Then $T_S$ is not big.
\end{Prop}

\begin{proof}
The surface $S$ can be obtained by blowing up $7$ bubble points $p_1,\,\ldots,p_7$ in $\mathbb P^2$ such that
$p_4\succ p_1$, $p_5\succ p_2$, $p_6\succ p_3$, and $\{p_1,\,p_3,\,p_4\}$, $\{p_2,\,p_3,\,p_5\}$, $\{p_3,\,p_6,\,p_7\}$, $\{p_1,\,p_2,\,p_7\}$ are collinear.

\vspace{-0.5em} 
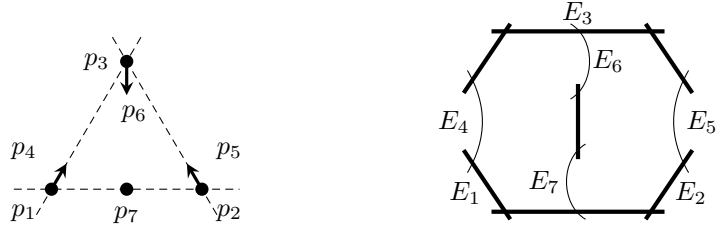
\begin{figure}[ht]
\begin{tikzpicture}[scale=1]
\begin{scope}[xshift=-3cm]
\node[circle,fill=black,inner sep=0pt,minimum size=5pt] at (0,0) {};
\node[below left=3pt of {(0,0)}] {$p_1$};
\node[above left=3pt of {(0,0.2)}] {$p_4$};
\node[circle,fill=black,inner sep=0pt,minimum size=5pt] at (2,0) {};
\node[below right=3pt of {(2,0)}] {$p_2$};
\node[above right=3pt of {(2,0.2)}] {$p_5$};
\node[circle,fill=black,inner sep=0pt,minimum size=5pt] at (1,1.7) {};
\node[left=3pt of {(1,1.7)}] {$p_3$};
\node[left=3pt of {(1.5,1)}] {$p_6$};
\node[circle,fill=black,inner sep=0pt,minimum size=5pt] at (1,0) {};
\node[below=3pt of {(1,0)}] {$p_7$};
\draw[very thick,color=black,-stealth] (0,0)--(0.2,0.34);
\draw[very thick,color=black,-stealth] (2,0)--(1.8,0.34);
\draw[very thick,color=black,-stealth] (1,1.7)--(1,1.25);

\draw[densely dashed,color=black] (-0.5,0)--(2.5,0);
\draw[densely dashed,color=black] (-0.2,-0.34)--(1.2,2.04);
\draw[densely dashed,color=black] (2.2,-0.34)--(0.8,2.04);
\end{scope}

\begin{scope}[xshift=3cm,yshift=-0.3cm]
\draw[ultra thick] (-0.15,0)--(2.15,0);
\draw[ultra thick] (-0.15,2.4)--(2.15,2.4);

\draw[ultra thick] (0.1,-0.1)--(-0.7-0.1 +0.28,1.2+0.1 -0.48);
\draw[ultra thick] (-0.7-0.1 +0.28,1.2-0.1 +0.48)--(0.1,2.4+0.1);
\draw[ultra thick] (2-0.1,-0.1)--(2+0.7+0.1 -0.28,1.2+0.1 -0.48);
\draw[ultra thick] (2+0.7+0.1 -0.28,1.2-0.1 +0.48)--(2-0.1,2.4+0.1);
\draw (-0.7-0.05 +0.28,1.2-0.2 -0.48) to[out=60,in=-60] (-0.7-0.05 +0.28,1.2+0.2 +0.48);
\draw (2+0.7+0.05 -0.28,1.2-0.2 -0.48) to[out=120,in=-120] (2+0.7+0.05 -0.28,1.2+0.2 +0.48);

\draw (1-0.1,2.4+0.1) to[out=-30,in=30] (1-0.1,1.6-0.1);
\draw (1+0.1,0-0.1) to[out=150,in=-150] (1+0.1,0.8+0.1);
\draw[ultra thick] (1,0.7)--(1,1.7);

\node at (-0.5,0.25) {$E_1$};
\node at (2.5,0.25) {$E_2$};
\node at (1,2.65) {$E_3$};
\node at (-0.65,1.2) {$E_4$};
\node at (2.65,1.2) {$E_5$};
\node at (1.4,2) {$E_6$};
\node at (0.55,0.4) {$E_7$};
\end{scope}
\end{tikzpicture}
\vspace{-1em} 
\caption{Configuration of $4$ lines and $(-2)$-curves of type $2A_3+A_1$ (degree $2$)}
\label{fig:2A_3+A_1}
\vspace{-0.5em} 
\end{figure}

\noindent
Then $K_S=-3H+E_1+E_2+E_3+2E_4+2E_5+2E_6+E_7.$. 
We define a set of effective divisors on $\PP(T_S)$ as follows.

\begin{itemize}
\item
$\breve\Cc$ is the total dual VMRT associated to the family of strict transforms $\ell$ of lines passing through $p_7$.
\item
$\breve\Dd_{1245}$ (resp. $\breve\Dd_{1236}$) is the total dual VMRT associated to the family of strict transforms $C_{1245}$ (resp. $C_{1236}$) of conics passing through $p_1,\,p_2,\,p_4,\,p_5$ (resp. $p_1,\,p_2,\,p_3,\,p_6$).
\end{itemize}
Then $[\ell]=H-E_7$ with the non-reduced curve $(H-E_3-2E_6-E_7)+2E_6$, 
$[C_{1245}]=2H-E_1-E_2-2E_4-2E_5$ with the non-reduced curves $2(H-E_1-E_2-E_4-E_5-E_7)+E_1+E_2+2E_7$, $(H-E_1-E_3-2E_4-E_6)+(H-E_2-E_3-2E_5-E_6)+2E_3+2E_6$, and $[C_{1236}]=2H-E_1-E_2-E_3-E_4-E_5-2E_6$ with the non-reduced curve $(H-E_1-E_2-E_4-E_5-E_7)+(H-E_3-2E_6-E_7)+2E_7$.
So we have 
\begin{align*}
\breve\Cc_7
&=\zeta+\Pi^*(K_S+2[\ell])-\Pi^*E_6\\
&=\zeta-3\Pi^*H+\hspace{0.5em}\Pi^*E_1+\hspace{0.5em}\Pi^*E_2+\hspace{0.5em}\Pi^*E_3+2\Pi^*E_4+2\Pi^*E_5+2\Pi^*E_6+\hspace{0.5em}\Pi^*E_7\\
&\hspace{1.8em}+2\Pi^*H\hspace{24.7em}-2\Pi^*E_7\\
&\hspace{25.9em}\hspace{0.4em}-\hspace{0.5em}\Pi^*E_6\\
&={\zeta-\hspace{0.5em}\Pi^*H+\hspace{0.5em}\Pi^*E_1+\hspace{0.5em}\Pi^*E_2+\hspace{0.5em}\Pi^*E_3+2\Pi^*E_4+2\Pi^*E_5+\hspace{0.5em}\Pi^*E_6-\hspace{0.5em}\Pi^*E_7},\\
\breve\Dd_{1245}
&=\zeta+\Pi^*(K_S+2[C_{1245}])-\Pi^*(H-E_1-E_2-E_4-E_5-E_7)-\Pi^*E_7\\
&=\zeta-3\Pi^*H+\hspace{0.5em}\Pi^*E_1+\hspace{0.5em}\Pi^*E_2+\hspace{0.5em}\Pi^*E_3+2\Pi^*E_4+2\Pi^*E_5+2\Pi^*E_6+\hspace{0.5em}\Pi^*E_7\\
&\hspace{1.8em}+4\Pi^*H-2\Pi^*E_1-2\Pi^*E_2\hspace{4.2em}-4\Pi^*E_4-4\Pi^*E_5\\
&\hspace{1.8em}-\hspace{0.5em}\Pi^*H+\hspace{0.5em}\Pi^*E_1+\hspace{0.5em}\Pi^*E_2\hspace{4.2em}+\hspace{0.5em}\Pi^*E_4+\hspace{0.5em}\Pi^*E_5\hspace{4.2em}+\hspace{0.5em}\Pi^*E_7\\
&\hspace{13.5em}
\hspace{0.4em}-\hspace{0.5em}\Pi^*E_3\hspace{8.3em}
-\hspace{0.5em}\Pi^*E_6
-\hspace{0.5em}\Pi^*E_7\\
&={\zeta\hspace{16.2em}-\hspace{0.5em}\Pi^*E_4-\hspace{0.5em}\Pi^*E_5+\hspace{0.5em}\Pi^*E_6+\hspace{0.5em}\Pi^*E_7},\\
\breve\Dd_{1236}
&=\zeta+\Pi^*(K_S+2[C_{1236}])-\Pi^*E_7\\
&=\zeta-3\Pi^*H+\hspace{0.5em}\Pi^*E_1+\hspace{0.5em}\Pi^*E_2+\hspace{0.5em}\Pi^*E_3+2\Pi^*E_4+2\Pi^*E_5+2\Pi^*E_6+\hspace{0.5em}\Pi^*E_7\\
&\hspace{1.8em}+4\Pi^*H-2\Pi^*E_1-2\Pi^*E_2-2\Pi^*E_3-2\Pi^*E_4-2\Pi^*E_5-4\Pi^*E_6\\
&\hspace{30em}
\hspace{0.4em}-\hspace{0.5em}\Pi^*E_7\\
&={\zeta+\hspace{0.5em}\Pi^*H-\hspace{0.5em}\Pi^*E_1-\hspace{0.5em}\Pi^*E_2-\hspace{0.5em}\Pi^*E_3\hspace{8.3em}-2\Pi^*E_6},
\end{align*}
which are summed up to $\breve\Cc_7+\breve\Dd_{1245}+\breve\Dd_{1236}=3\zeta+\Pi^*E_4+\Pi^*E_5$.
Thus $T_S$ is not big by Lemma~\ref{not_big_lemma}.
\end{proof}

\begin{Rmk}
Consider Figure~\ref{fig:2A_3+A_1} in Proposition~\ref{2A_3+A_1}.
If we blow up at a smooth point of $E_4$ or $E_5$, then we have a weak del Pezzo surface $S$ of degree $1$ of type $A_7+A_1$.
If we blow up at a smooth point of $E_6$ or $E_7$, then we have a weak del Pezzo surface $S$ of degree $1$ of type $A_3+D_5$.
Thus $T_S$ is not big for both cases by Proposition~\ref{2A_3+A_1}. 
\end{Rmk}

\begin{Rmk}
Consider a weak del Pezzo surface $S$ of degree $3$ of type $D_4$.
By Lemma~\ref{D_4}, $T_S$ is not big.
We then consider \cite[Figure 16]{MS}.
If we blow up at the intersection point of three $(-1)$-curves then we have a weak del Pezzo surface $S$ of degree $2$ of type $3A_1+D_4$ \cite[Figure 9]{MS}.
Then we obtain a weak del Pezzo surface $S$ of degree $1$ of type $2D_4$ or $E_6+2A_1$ by blowing up at a smooth point in $(-1)$-curve.
Thus $T_S$ is not big for all of these examples.
\end{Rmk}

We conclude this section with the following natural question: Is there a weak del Pezzo surface $S$ of degree $2$ or $1$ with big $T_S$?


\begin{thebibliography}{11}

\bibitem{BHPV}
    W. Barth, K. Hulek, C. Peters, Chris, A. Van de Ven, 
    Compact complex surfaces. 2nd edition. Ergebnisse der Mathematik und ihrer Grenzgebiete. 3. Folge. A Series of Modern Surveys in Mathematics, Springer-Verlag, Berlin, 2004. 

\bibitem{Der14}
    U. Derenthal,
    \textit{Singular del Pezzo surfaces whose universal torsors are hypersurfaces},
    Proc. Lond. Math. Soc. (3) \textbf{108} (2014), no. 3,
    638--681.

\bibitem{Dol12}
    I. Dolgachev,
    {Classical algebraic geometry. A modern view},
    Cambridge University Press, Cambridge,
    2012.

\bibitem{DK}
    I. Dolgachev, S. Kondo,
    Enriques surfaces II,
    Springer (to appear). 
    
\bibitem{HLS22}
    A. H\"{o}ring, J. Liu, F. Shao,
    \textit{Examples of Fano manifolds with non-pseudoeffective tangent bundle},
    J. Lond. Math. Soc. (2) \textbf{106} (2022), no. 1,
    27--59.

\bibitem{HP}
    A. H\"{o}ring, T. Peternell,
    \textit{Stein complements in compact K\"ahler manifolds},
    Math. Ann (to appear),
    \href{https://arxiv.org/abs/2111.03303}{arXiv:2111.03303}.
    
\bibitem{Hsiao15} 
    J.-C. Hsiao,
    \textit{A remark on bigness of the tangent bundle of a smooth projective variety and $D$-simplicity of its section rings}, 
    J. Algebra Appl. \textbf{14} (2015), no. 7,
    1550098.   

\bibitem{HR04}
    J.-M. Hwang, S. Ramanan,
    \textit{Hecke curves and Hitchin discriminant},
    Ann. Sci. École Norm. Sup. (4) \textbf{37} (2004), no. 5,
    801--817.

\bibitem{JLZ}
    J. Jia, Y. Lee, G. Zhong,
    \textit{Smooth projective surfaces with pseudo-effective tangent bundles},
    J. Math. Soc. Japan (to appear),
    \href{https://arxiv.org/abs/2302.10077}{arXiv:2302.10077}.

\bibitem{KKL}
    H. Kim, J.-S. Kim, Y. Lee,
    \textit{Bigness of the tangent bundle of a Fano threefold with Picard number two},
    \href{https://arxiv.org/abs/2201.06351}{arXiv:2201.06351}.

\bibitem{Kim23}
    J.-S. Kim,
    \textit{Bigness of the tangent bundles of projective bundles over curves}, 
    C. R. Math. Acad. Sci. Paris \textbf{361} (2023),
    1115--1122.
    
\bibitem{KM98}
    J. Koll\'ar, S. Mori, 
    {Birational geometry of algebraic varieties},
    with the collaboration of C. H. Clemens and A. Corti,
    Cambridge Tracts in Mathematics \textbf{134},
    Cambridge University Press, Cambridge, 1998.

\bibitem{Mal21}
    D. Mallory,
    \textit{Bigness of the tangent bundle of del Pezzo surfaces and $D$-simplicity}, 
    Algebra Number Theory \textbf{15} (2021), no. 8,
    2019--2036.

\bibitem{MS}
    G. Martin, C. Stadlmayr,
    \textit{Weak del Pezzo surfaces with global vector fields},
    Geom. Topol. (to appear),
    \href{https://arxiv.org/abs/2007.03665}{arXiv:2007.03665}.

\bibitem{Mar82}
    M. Maruyama,
    \textit{Elementary transformations in the theory of algebraic vector bundles},
    in Algebraic Geometry (La R\'abida, 1981),
    Lecture Notes in Math. \textbf{961}, 241--266,
    Springer, Berlin, 1982.

\bibitem{OSW16}
    G. Occhetta, L. E. Sol\'a Conde, K. Watanabe,
    \textit{Uniform families of minimal rational curves on Fano manifolds},
    Rev. Mat. Complut. \textbf{29} (2016), no. 2,
    423--437.

\bibitem{RR13}
    X. Roulleau, R. Rousseau,
    \textit{On the hyperbolicity of surfaces of general type with small $c_1^2$},
    J. Lond. Math. Soc. \textbf{87} (2013), no. 2,
    453--477.

\bibitem{San14}
    T. Sano,
    \textit{Seshadri constants on rational surfaces with anticanonical pencils},
    J. Pure Appl. Algebra \textbf{218} (2014) no. 4,
    602--617.

\bibitem{Ser92}
    F. Serrano, 
    \textit{Fibred surfaces and moduli},
    Duke Math. J. \textbf{67} (1992) no. 2,
    407--421. 
    
\bibitem{Ser96}
    F. Serrano, 
    \textit{Isotrivial fibred surfaces},
    Ann. Mat. Pura Appl. \textbf{171} (1996) no. 4,
    63--81.
    
\bibitem{Ura83}
    T. Urabe,
    \textit{On singularities on degenerate del Pezzo surfaces of degree $1$, $2$},
    Singularities, Part 2 (Arcata, Calif., 1981),
    587--591,
    Proc. Sympos. Pure Math. \textbf{40},
    Amer. Math. Soc., Providence, RI, 1983.

\end{thebibliography}
\end{document}